\documentclass{amsart}

\usepackage{amsmath,amssymb, amsfonts,amscd,amsxtra}
\usepackage{amsthm}
\usepackage[nobysame]{amsrefs}
\usepackage{indentfirst}
\usepackage{mathrsfs}
\usepackage{multirow}
\usepackage{enumerate}
\usepackage{graphicx}
\usepackage{hyperref}
\usepackage{xcolor}
\usepackage{fourier}
\usepackage{float}
\usepackage[norelsize,boxed,linesnumbered,vlined,algo2e]{algorithm2e} 

\theoremstyle{plain}
\newtheorem{theorem}{Theorem}

\newtheorem{prop}[theorem]{Proposition}

\theoremstyle{definition}

\theoremstyle{remark}
\newtheorem*{remark}{Remark}

\DeclareMathOperator{\tr}{tr}

\newcommand{\RR}{\mathbb{R}}

\newcommand{\TT}{\mathrm{T}}

\newcommand{\Or}{\mathcal{O}}

\newcommand{\wh}[1]{\widehat{#1}}

\def\E{\mathcal{E}}
\def\C{\mathcal{C}}
\def\B{\mathcal{B}}

\def\RR{\mathbb{R}}
\def\T{\mathfrak{T}}

\IfFileExists{mathabx.sty}%
  {\DeclareFontFamily{U}{mathx}{\hyphenchar\font45}%
   \DeclareFontShape{U}{mathx}{m}{n}{<->mathx10}{}%
   \DeclareSymbolFont{mathx}{U}{mathx}{m}{n}%
   \DeclareFontSubstitution{U}{mathx}{m}{n}%
   \DeclareMathAccent{\widebar}{0}{mathx}{"73}%
   
}{%
  \PackageWarning{mathabx}{%
    Package mathabx not available, therefore\MessageBreak substituting
    widebar with overline\MessageBreak }%
  \newcommand{\widebar}[1]{\overline{#1}}%
}

\newcommand{\abs}[1]{\left\lvert#1\right\rvert}
\newcommand{\norm}[1]{\left\lVert#1\right\rVert}
\newcommand{\average}[1]{\langle#1\rangle}

\newcommand{\vertiii}[1]{\VERT #1 \VERT}

\begin{document}

\title{Localized density matrix minimization and linear scaling
  algorithms} 

\thanks{
The research of J.L.~was supported in part by the Alfred P.~Sloan
  Foundation and the National Science Foundation under award
  DMS-1312659.  
  The authors would like to thank Stanley Osher and Vidvudz Ozolins for their encouragements and helpful discussions.
  }

\author{Rongjie Lai}
\address{Department of Mathematics, Rensselaer Polytechnic Institute. }
\email{lair@rpi.edu}

\author{Jianfeng Lu}
\address{Departments of Mathematics, Physics, and Chemistry, Duke University}
\email{jianfeng@math.duke.edu}

\date{}

\begin{abstract}
  We propose a convex variational approach to compute localized
  density matrices for both zero temperature and finite temperature
  cases, by adding an entry-wise $\ell_1$ regularization to the free
  energy of the quantum system. Based on the fact that the density
  matrix decays exponential away from the diagonal for insulating
  system or system at finite temperature, the proposed $\ell_1$
  regularized variational method provides a nice way to approximate
  the original quantum system. We provide theoretical analysis of the
  approximation behavior and also design convergence guaranteed
  numerical algorithms based on Bregman iteration. More importantly,
  the $\ell_1$ regularized system naturally leads to localized density
  matrices with banded structure, which enables us to develop
  approximating algorithms to find the localized density matrices with
  computation cost linearly dependent on the problem size.
\end{abstract}

\maketitle

\section{Introduction}


Efficient calculation of the low-lying spectrum of operators plays a
central role in many applications. In particular, in the context of
electronic structure theory, given a discretized effective Hamiltonian
(such as the current iterate in a self-consistent iteration), the goal
is to obtain the density matrix corresponds to the number of
electrons. For zero temperature, the density matrix is the projection
operator onto the low-lying eigenspace; for finite temperature, the
density matrix is given by the Fermi-Dirac function acting on the
Hamiltonian \cites{ParrYang:89, Martin:04}.

In this work, we extend the variational approach for localized density
matrix in our previous work \cite{LaiLuOsher:15} to finite
temperature, by adding entrywise $\ell_1$ penalty to the free energy
of the quantum system (which, in the context of density functional
theory, corresponds to the linear version of the Mermin functional \cite{Mermin:65}).
We also theoretically show that the proposed localized density matrix approximates, via the Frobenius norm, the true
density matrix linearly depends on the regularization parameter $1/\eta$. In addition, convergence guaranteed numerical algorithms are also
designed to solve the proposed problems based on Bregman iteration. 

More importantly, this paper focuses on efficient algorithms to
minimize the variational problem for localized density matrix both at
zero and finite temperature. In particular, we develop linear scaling
algorithms such that the computational cost scales linearly as the
dimension of the matrix. The key idea is to exploit the decay of the
density matrix away from the diagonal, the $\ell_1$ regularized localized density matrices enable us to approximate the
original variational problem by restricting to banded matrices.

Linear scaling algorithms have been a focused research direction in
electronic structure calculation since 1990s.  Closely related to our
context is the density matrix minimization (DMM) algorithms, first
introduced by Li, Nunes, and Vanderbilt \cite{LiNunesVanderbilt:93}
and Daw \cite{Daw:93}, and have been further developed since then, see
e.g., the reviews \cites{Goedecker:99, BowlerMiyazaki:12}.  These
algorithms are based on the fact that for insulating system or systems
at finite temperature, the density matrix decays exponentially away
from the diagonal (see e.g., \cites{Kohn:96, ELu:ARMA,
  Benzi:13}). Thus, one may require the density matrices to satisfy
prescribed sparsity structures, such as banded matrices for 1D
problem. As a result, the degree of freedom and the computational cost
becomes linear scaling.  Another closely related class of methods is
the purification method of density matrix (see e.g.,
\cites{McWeeny:60, PalserManolopoulos:98, Niklasson:02, Mazziotti:03}
and a review \cite{Niklasson:11}). Unlike the approach of DMM and that
we take in this work, these methods are not variational.

We emphasize a crucial difference between our approach and the
previous works: \emph{Our variational problem is still convex even
  after truncation!}  This is in stark contrast to the previous
formulations where the convexity is lost by using purification
\cite{McWeeny:60} or other techniques to approximate the inverse of
density matrix \cite{Challacombe:99}. This loss of convexity often
introduces local minimizers to the variational problem and also issues
of convergence.  We note that even when the $\ell_1$ regularization is
dropped from our variational problem, it is still convex and different
from the standard DMM algorithms. In fact, it would be of interest to
explore this convex formulation, which will be considered in our
future work.

The rest of the paper is organized as follows. In the next section, we
introduce the variational principles for localized density matrix for
both zero and finite temperature cases, with their approximation
properties. In Section~\ref{sec:bregmanforLDM}, we introduce the
Bregman iteration type algorithms to solve these minimization
problems. Linear scaling algorithms are discussed in
Section~\ref{sec:linear}.  We validate the algorithms through
numerical examples in Section~\ref{sec:experiments}. Some conclusive
remarks are discussed in Section~\ref{sec:conclusion}.

\section{Localized density matrix minimization}


In this work, we will consider a family of energy functionals with
parameters $\beta$ and $\eta$:
\begin{equation}
  \E_{\beta, \eta} = \tr (H P) + \frac{1}{\beta} \tr\Big\{P\ln P + (1- P)\ln(1 - P)\Big\} + \frac{1}{\eta}\vertiii{P}_1
\end{equation}
where $\vertiii{\cdot}$ denotes the entrywise $\ell_1$ norm of a
matrix and $H$ and $P$ are $n\times n$ symmetric matrices, which are
respectively the (discrete) Hamiltonian and density matrix in the
context of electronic structure calculation. Here $\beta$ is the
inverse temperature and $\eta$ is the parameter for the $\ell_1$
regularization.

Recall that the starting point of \cite{LaiLuOsher:15} is the convex
variational principle for density matrix 
\begin{equation}\label{eqn:DM_0T}
  \begin{aligned}
    & \min_{P \in \RR^{n\times n}} \E_{\infty, \infty}(P) = \min_{P \in \RR^{n\times n}} \tr(H P), \\
    & \text{s.t.} \quad \tr P = N,\; P = P^{\TT},\; 0 \preceq P \preceq I,
  \end{aligned}
\end{equation}
where the notation $A \preceq B$ denotes that $B - A$ is a symmetric
positive semi-definite matrix.  Note that the constraint $0 \preceq P
\preceq I$ is the convexification of the idempotency constraint $P =
P^2$, which gives the same minimizer for non-degenerate problems (see
e.g., \cite{LaiLuOsher:15}*{Proposition 1}). Indeed, denote
$\{\lambda_i, \phi_i\}_{i=1}^n$ the eigenvalue and eigenvector pairs
of $H$ with the assumption $\lambda_N < \lambda_{N+1}$, the solution
of \eqref{eqn:DM_0T} is given by
\begin{equation}
  P_{\infty, \infty} = \sum_{i=1}^N \phi_i \phi_i^{\TT},
\end{equation}
the projection operator on the subspace spanned by the first $N$
eigenvectors.

The variational principle \eqref{eqn:DM_0T} corresponds to the
physical zero temperature case (and hence the inverse temperature
$\beta = \infty$), for finite temperature, we minimize the free energy
\begin{equation}\label{eqn:DM_FT}
  \begin{aligned}
    & \min_{P \in \RR^{n \times n}} \E_{\beta, \infty}(P) = \min_{P \in \RR^{n \times n}} \tr ( H P) + \frac{1}{\beta} \tr\Bigl\{P\ln P + (1- P)\ln(1 - P)\Bigr\} \\
    & \text{s.t.}  \quad  \tr P = N,\; P = P^{\TT}, \; 0 \preceq P \preceq I,
  \end{aligned}
\end{equation}
where we have used the Fermi-Dirac entropy
\begin{equation}
  \varphi(x) = x \ln x + (1 - x) \ln (1 - x), \qquad x \in [0, 1].
\end{equation}
Note that 
\begin{align}
  & \varphi'(x) = \ln x - \ln (1 - x), \\
  & \varphi''(x) = x^{-1} + (1 - x)^{-1}, 
\end{align}
and hence $\varphi''(x) \geq 4$ for $x \in [0, 1]$. Therefore, $\tr
\varphi(P)$ is strictly convex with respect to $P$ and the minimizer of \eqref{eqn:DM_FT} exists and is uniquely given by 
\begin{equation}\label{eq:Pbetainf}
  P_{\beta, \infty} = \Bigl[ 1 + \exp( \beta (H - \mu)) \Bigr]^{-1}  
  = \sum_{i=1}^n \rho_i \phi_i \phi_i^{\TT},
\end{equation}
where $\rho_i$ is the occupation number of the $i$-th eigenstate, given by 
\begin{equation*} 
  \rho_i = \frac{1}{1 + \exp(\beta ( \lambda_i
    - \mu))} \in [0, 1], \qquad i=1, \ldots, n,
\end{equation*} 
and $\mu$ is the Lagrange multiplier for the constraint $\tr P = N$,
known as the chemical potential. It is determined by
\begin{equation*}
  \sum_{i=1}^n \frac{1}{1 + \exp(\beta ( \lambda_i - \mu))} = N.
\end{equation*}
For a fixed Hamiltonian matrix, $\mu$ is then a function of
$\beta$. It is not difficult to see that 
\begin{equation}
  \lim_{\beta \to \infty} \mu(\beta) = \frac{1}{2} (\lambda_N + \lambda_{N+1}), 
\end{equation}
which lies in between the highest occupied and lowest unoccupied
eigenvalues.


In \cite{LaiLuOsher:15}, the following $\ell_1$ regularized version of
the variaional principle \eqref{eqn:DM_0T} is proposed as a convexified
model for compressed modes introduced in
\cite{OzolinsLaiCaflischOsher:13}.
\begin{equation}\label{eqn:LDM_0T}
  \begin{aligned}
    P_{\infty,\eta} = & \arg\min_{P \in \RR^{n \times n}} \E_{\infty, \eta} =  \min_{P \in \RR^{n\times n}} \tr(H P) + \frac{1}{\eta} \vertiii{P}_1, \\
    & \text{s.t.} \quad \tr P = N,\; P = P^{\TT}, \; 0 \preceq P
    \preceq I,
  \end{aligned}
\end{equation}
where $\vertiii{\cdot}_1$ is the entrywise $\ell_1$ matrix norm and
$\eta$ is a penalty parameter for entrywise sparsity. The variational
principle provides $P_{\infty,\eta}$ as a sparse representation of the
projection operator onto the low-lying eigenspace. Numerical
experiments in \cite{LaiLuOsher:15} demonstrate the localization of
$P_{\infty,\eta}$ for the above $\ell_1$ regularized model.

For the finite temperature, applying the same $\ell_1$ regularization to enhance sparsity, we arrive at 
\begin{equation}\label{eqn:LDM_FT}
  \begin{aligned}
    & \min_{P \in \RR^{n \times n}} \E_{\beta, \eta}(P) 
    = \min_{P \in \RR^{n\times n}} \tr(H P) + \frac{1}{\beta} \tr\Big\{P\ln P + (1- P)\ln(1 - P)\Big\} + \frac{1}{\eta}\vertiii{P}_1\\
    & \text{s.t.}  \quad \tr P = N, \; P = P^{\TT}, \;  0 \preceq P \preceq I.
  \end{aligned}
\end{equation}
Since this variational principle is strictly convex, the minimizer
exists and is unique.

We will call the minimizing density matrix obtained by the above
variational principles the {\em localized density matrix} (LDM).  We
provide in the remaining of this section some approximation results of
LDM as the parameter $\eta \to \infty$.

\begin{theorem}\label{thm:LDM_0T_consist} 
  Assume $H$ is non-degenerate that $\lambda_N < \lambda_{N+1}$ and
  denote $P_{\infty, \infty}$ the minimizer of \eqref{eqn:DM_0T}. Let
  $P_{\infty, \eta}$ be a minimizer of \eqref{eqn:LDM_0T}, we have
  \begin{enumerate}\renewcommand{\labelenumi}{(\alph{enumi})}
  \item $\displaystyle 0\leq \E_{\infty, \infty}(P_{\infty, \eta}) -
    \E_{\infty, \infty}(P_{\infty, \infty}) \leq \frac{1}{\eta}
    \vertiii{P_{\infty, \infty}}_1$.
  \item $\displaystyle \| P_{\infty, \eta} - P_{\infty, \infty} \|_F^2 \leq
    \frac{2}{\eta}\frac{\vertiii{P_{\infty, \infty}}_1}{ (\lambda_{N+1} -
      \lambda_{N})}$.
  \end{enumerate}
  In particularly, $\displaystyle \lim_{\eta\rightarrow\infty}
  \E_{\infty, \infty}(P_{\infty, \eta}) = \E_{\infty,
    \infty}(P_{\infty, \infty})$, and $\displaystyle
  \lim_{\eta\rightarrow\infty} \norm{ P_{\infty, \eta} - P_{\infty, \infty} }_F = 0$.
\end{theorem}

\begin{remark}
  Recall that the minimizer of \eqref{eqn:LDM_0T} might not be unique
  \cite{LaiLuOsher:15}, nevertheless the theorem applies to any
  minimizers.
\end{remark}

\begin{proof} 
  It is clear that $\E_{\infty, \infty}(P_{\infty,
    \infty}) \leq \E_{\infty, \infty}(P_{\infty, \eta})$ as
  $P_{\infty, \infty}$ minimizes the energy $\E_{\infty, \infty}$. On
  the other hand,since $P_\eta$ is an optimizer of \eqref{eqn:LDM_0T},
  we have,
  \begin{equation*}
    \E_{\infty, \infty}(P_{\infty, \eta}) \leq   \E_{\infty, \infty}(P_{\infty, \eta}) + \frac{1}{\eta}\vertiii{P_{\infty, \eta}}_1 \leq \E_{\infty, \infty}(P_{\infty, \infty}) +  \frac{1}{\eta} \vertiii{P_{\infty, \infty}}_1
  \end{equation*}
  which yields the statement (a). 

  To show (b), we denote $\{\lambda_i, \phi_i\}_{i=1}^n$ the
  eigenpairs of $H$ and recall that 
  \begin{equation*}
    P_{\infty, \infty} = \sum_{1 \leq i \leq N} \phi_i \phi_i^{\TT}.
  \end{equation*} 
  Moreover, as $P_{\infty, \eta}$ is symmetric and hence
  diagonalizable, we denote $\{\sigma_i, v_i\}_{i=1}^n$ its
  eigenpairs, with $P_{\infty, \eta} v_i = \sigma_i v_i$. Note that
  $\sum_i \sigma_i = N$ and $\sigma_i \in [0, 1]$ since $P_{\infty,
    \eta}$ satisfies the constraint of \eqref{eqn:LDM_0T}.
  
  Using the property of trace,
  we calculate
  \begin{equation*}
    \tr(H P_{\infty, \eta}) = \sum_{i=1}^n \average{\phi_i, H
      P_{\infty, \eta} \phi_i} = \sum_{i=1}^n \lambda_i \average{\phi_i, P_{\infty, \eta} \phi_i} =: \sum_{i=1}^n \lambda_i s_i, 
  \end{equation*}
  where the last equality defines the shorthand notation $s_i =
  \average{\phi_i, P_{\infty, \eta} \phi_i}$. Since $0 \preceq
  P_{\infty, \eta} \preceq I$ and $\tr P_{\infty, \eta} = N$, we have
  \begin{equation*} 
    0\leq s_i \leq 1, \qquad\text{and}\qquad \sum_{i=1}^n s_i = \sum_{i=1}^n
    \average{\phi_i, P_{\infty, \eta}\phi_i} = \tr P_{\infty, \eta} = N. 
  \end{equation*}
  We now estimate, based on these properties of $\{s_i\}$,  
  \begin{multline*}
    \frac{1}{\eta} \vertiii{P_{\infty, \infty}}_1 \geq \E_{\infty,
      \infty}(P_{\infty, \eta}) - \E_{\infty, \infty}(P_{\infty,
      \infty}) \\
    = \sum_{j=1}^n\lambda_j s_j - \sum_{j=1}^N \lambda_j \geq
    \sum_{j=1}^N \lambda_j (s_j - 1) + \sum_{j=N+1}^n \lambda_j s_j \\
    \geq \lambda_N \sum_{j=1}^N (s_j - 1) + \lambda_{N+1}
    \sum_{j=N+1}^n s_j = (\lambda_{N+1} - \lambda_N) \sum_{j=1}^N (1 -
    s_j ).
  \end{multline*}
  This yields
  \begin{equation}
    \label{eqn:LDM_b}
    N - \sum_{j=1}^N s_j = \sum_{j =1 }^N (1 - s_j ) \leq  \frac{1}{\eta}\frac{\vertiii{P_{\infty, \infty}}_1}{ (\lambda_{N+1} - \lambda_{N})}.
  \end{equation}
  We complete the proof of (b) by
  \begin{align*}
    \norm{P_{\infty, \eta} - P_{\infty, \infty}}^2_F & = \tr\Bigl(
    (P_{\infty, \eta} - P_{\infty, \infty})^2\Bigr) \\
    & = \tr(P_{\infty, \eta}^2) - 2 \tr(P_{\infty,\infty} P_{\infty, \eta}) + \tr(P_{\infty,\infty}^2)  \\
    & \leq \tr(P_{\infty, \eta}) - 2  \tr(P_{\infty,\infty} P_{\infty, \eta}) + \tr(P_{\infty,\infty})\\
& = 2 N - 2 \sum_{i=1}^n \average{\phi_i, P_{\infty, \infty} P_{\infty, \eta} \phi_i} \\ 
& = 2\Bigl(N - \sum_{j=1}^N s_j \Bigr) \leq  \frac{2}{\eta}\frac{\vertiii{P_{\infty, \infty}}_1}{ (\lambda_{N+1} - \lambda_{N})}.
\end{align*}
\end{proof}

For the finite temperature case, we have the following analogous
result.
\begin{theorem}\label{thm:finitetemp}
  Denote $P_{\beta, \infty}$ and $P_{\beta, \eta}$ the minimizer of
  \eqref{eqn:DM_FT} and \eqref{eqn:LDM_FT} respectively, it holds 
  \begin{equation}\label{eq:finitetempenergy}
    0 \leq \E_{\beta, \infty}(P_{\beta, \eta}) - \E_{\beta, \infty}(P_{\beta, \infty})
    \leq \frac{1}{\eta} \vertiii{P_{\beta, \infty}}_1, 
  \end{equation}
  and also the estimate 
  \begin{equation}\label{eq:finitetempfrob}
    \tr \Bigl(\max\bigl( \beta^{-1}, \abs{H - \mu} \bigr) ( P_{\beta, \eta} - P_{\beta, \infty})^2\Bigr) \leq \frac{1}{\eta} \vertiii{P_{\beta, \infty}}_1. 
  \end{equation}
\end{theorem}

\begin{remark}
  Note that as immediate consequence of \eqref{eq:finitetempfrob}, we have
  the estimate in Frobenius norm
  \begin{equation}
    \norm{P_{\beta, \eta} - P_{\beta, \infty}}_F^2 \leq \frac{1}{\eta} \vertiii{P_{\beta, \infty}}_1 \min\bigl( \beta, \max_i \, \abs{\lambda_i - \mu}^{-1}\bigr). 
  \end{equation}
  Taking the limit $\beta \to \infty$, as the chemical potential $\mu
  \to \frac{1}{2} (\lambda_N + \lambda_{N+1})$, we get 
  \begin{equation*}
    \lim_{\beta \to \infty} \min_i \abs{\lambda_i - \mu} \to \frac{\lambda_{N+1} - \lambda_{N}}{2}.
  \end{equation*}
  Therefore, we recover the estimate for zero temperature case (assuming $\lambda_N < \lambda_{N+1}$).
\end{remark}

\begin{proof}
  By optimality of $P_{\beta, \eta}$ and $P_{\beta, \infty}$, we have
  $\E_{\beta, \infty}(P_{\beta, \infty}) \leq \E_{\beta,
    \infty}(P_{\beta, \eta})$ and
  \begin{multline}\label{eq:thm2eq1}
    \frac{1}{\eta} \vertiii{P_{\beta, \infty}}_1 \geq \E_{\beta, \infty}(P_{\beta, \eta}) - \E_{\beta, \infty}(P_{\beta, \infty}) \\
    = \tr \bigl( (P_{\beta, \eta} - P_{\beta, \infty}) H\bigr) + \beta^{-1} \tr \bigl( \varphi(P_{\beta, \eta}) - \varphi(P_{\beta, \infty})\bigr).
  \end{multline}
  Hence, we obtain the first conclusion \eqref{eq:finitetempenergy} of
  the theorem.

  Recall that $\varphi'(x) = \ln x - \ln (1 - x)$ and hence by
  explicit calculation using \eqref{eq:Pbetainf}
  \begin{equation*}
    \varphi'(P_{\beta, \infty}) = \ln \Bigl( P_{\beta, \infty} ( I - P_{\beta, \infty})^{-1}\Bigr) = - \beta( H - \mu).
  \end{equation*}
  Therefore 
  \begin{equation*}
    \tr \bigl( (P_{\beta, \eta} - P_{\beta, \infty}) H \bigr) = \tr \bigl( (P_{\beta, \eta} - P_{\beta, \infty}) (H - \mu) \bigr) = - \beta^{-1} \tr \bigl( \varphi'(P_{\beta, \infty}) (P_{\beta, \eta} - P_{\beta, \infty})  \bigr),
  \end{equation*}
  where we have used that $\tr P_{\beta, \eta} = \tr P_{\beta, \infty}
  = N$ in the first equality. Substitute into \eqref{eq:thm2eq1}, we
  get 
  \begin{equation}\label{eq:thm2eq2}
       \frac{1}{\eta} \vertiii{P_{\beta, \infty}}_1 \geq \beta^{-1}
      \Bigl[ \tr \bigl( \varphi(P_{\beta, \eta}) - \varphi(P_{\beta,
        \infty})\bigr) - \tr \bigl( \varphi'(P_{\beta, \infty})
      (P_{\beta, \eta} - P_{\beta, \infty}) \bigr) \Bigr] 
   \end{equation}
   Note that the right hand side is the Bregman divergence of
   $\varphi$. For $x, y \in [0, 1]$, we have
   \begin{equation*}
     \varphi(x) - \varphi(y) - \varphi'(y)(x - y) 
     = x (\ln x - \ln y) + (1 - x) ( \ln (1 - x) - \ln (1 - y)),
   \end{equation*}
   which is the Fermi-Dirac relative entropy. Following a similar
   calculation in \cite{HainzlLewinSeiringer:08}*{Theorem 1} (see also
   an improved version in \cite{FrankHainzlSeiringerSolovej:12}), we
   minimize $x \in [0, 1]$ for fixed $y$ and find
   \begin{equation*}
     x \ln \frac{x}{y} + (1 - x) \ln\frac{1 - x}{1 - y} \geq 
     \frac{\ln\frac{1 - y}{y}}{1 - 2y} (x - y)^2. 
   \end{equation*}
   Explicit calculation verifies that for $y \in [0, 1]$, we have 
   \begin{equation*}
     \frac{\ln \frac{1-y}{y}}{1 - 2y} \geq \max \Bigl( \bigl\lvert
     \ln \frac{1 - y}{y} \bigr\rvert, 1\Bigr). 
   \end{equation*}
   Hence, combining with \eqref{eq:thm2eq2} and using Klein's lemma
   \cite{Ruelle:99}*{Theorem 2.5.2}, we arrived at the estimate \eqref{eq:finitetempfrob}
   \begin{equation*}
     \frac{1}{\eta} \vertiii{P_{\beta, \infty}}_1 \geq  
     \tr \Bigl( \max\bigl( \beta^{-1}, \abs{H - \mu} \bigr) ( P_{\beta, \eta} - P_{\beta, \infty})^2 \Bigr).
   \end{equation*}
\end{proof}

\section{Numerical algorithms for LDMs}
\label{sec:bregmanforLDM}

\subsection{Bregman iteration for LDMs}
In \cite{LaiLuOsher:15}, a numerical algorithm has been proposed to
solve \eqref{eqn:LDM_0T} based on Bregman iteration. Bregman iteration
was first introduced into information science for solving total
variation related problems as an analog of ``adding back the noise''
in image denoising~\cite{Osher:2005MMS}. Split Bregman iteration has
been later proposed in \cite{Goldstein:2009split} based on the idea of
variable splitting. These algorithms have since received intensive
attention due to its efficiency in many $\ell_1$ related constrained
optimization problems~\cites{Yin:2008bregman,yin2013error}. The
equivalence of the Bregman iteration with the alternating direction
method of multipliers (ADMM), Douglas-Rachford splitting and augmented
Lagrangian method can be found
in~\cites{Yin:2008bregman,Esser:2009CAM,Wu:2010SIAM,
  combettes2011proximal}.

Let's first recall the algorithm proposed in \cite{LaiLuOsher:15}. By introducing auxiliary variables $Q $ and $R$, the optimization
problem \eqref{eqn:LDM_0T} is equivalent to
\begin{equation}\label{eqn:LDM_0T_P_split}
  \begin{aligned}
    & \min_{P, Q, R \in \RR^{n\times n}} \frac{1}{\eta} \vertiii{Q}_1 + \tr(H P)\\
    & \text{s.t.} \  Q = P,\, R = P,\,  \tr P = N,\, 
    0 \preceq R \preceq I,
  \end{aligned}
\end{equation}
The method of Bregman iteration suggests to approach \eqref{eqn:LDM_0T_P_split} by solving:
\begin{align}
  & (P^k,Q^k,R^k) = \arg\min_{P, Q, R \in \RR^{n\times n}} \frac{1}{\eta} \vertiii{Q}_1 + \tr(H P) \label{eqn:LDM_0T_PQR}\\ 
  & \nonumber \hspace{8em} + \frac{\lambda}{2} \|P - Q + B^{k-1}\|_F^2 + \frac{r}{2}\|P - R + D^{k-1}\|^2_F \\   
  & \qquad   \qquad       \text{s.t.}  \qquad \tr P = N, \,  0 \preceq R \preceq I,   \nonumber \\
  & B^k = B^{k-1} + P^k - Q^k,   \\
  & D^k = D^{k-1} + P^k - R^k,
\end{align}
where variables $B, D$ are essentially Lagrangian multipliers and
parameters $r, \lambda$ control the penalty terms.  Solving
$P^k, Q^k, R^k$ in \eqref{eqn:LDM_0T_PQR} alternatively leads to
algorithm~\ref{alg:LDM_0T}, proposed in \cite{LaiLuOsher:15}.

\begin{algorithm2e}[h]
\caption{Zero temperature localized density matrix minimization}
\label{alg:LDM_0T}
Initialize $Q^0 = R^0 = P^0 \in\C = \{ P\in\RR^{n\times n}~|~ P = P^{\TT},\, \tr P = N,\, 0 \preceq P \preceq I \} , B^0 = D^0 = 0$

\While{``not converge"}{
  $\displaystyle P^k =\Gamma^k - \frac{\tr ( \Gamma^k) - N}{n},
  \quad  \text{where } 
  \Gamma^k = \frac{\lambda}{\lambda+r}(Q^{k-1} - B^{k-1}) +
  \frac{r}{\lambda+r}(R^{k-1} - D^{k-1}) - \frac{1}{\lambda +r} H $.
  
 $\displaystyle Q^k = \operatorname{Shrink}\left(P^k +B^{k-1},\frac{1}{\lambda\eta} 
\right) = \operatorname{sign}(P^k + B^{k-1})\max\left\{|P^{k} +B^{k-1}| - \frac{1}{\lambda\eta},0\right\} $.

 $R^k = \displaystyle V\min\{\max\{\Lambda,0\},1\}V^T , \text{ where } [V,~ \Lambda] =
  \operatorname{eig}(P^k + D^{k-1})$.
  
 $B^{k} = B^{k-1} +  P^k - Q^k$.
 
 $D^{k} = D^{k-1} +  P^k - R^k$.
}
\end{algorithm2e}

Similarly, the LDM for the finite temperature case proposed in
\eqref{eqn:LDM_FT} can also be solved based on Bregman iteration. By
introducing auxiliary variables $Q$ and $R$, the optimization problem
\eqref{eqn:LDM_FT} is equivalent to
\begin{equation}\label{eqn:LDM_FT_P_split}
  \begin{aligned}
    & \min_{P, Q, R \in \RR^{n\times n}} \frac{1}{\eta} \vertiii{Q}_1 + \tr(H P) + \frac{1}{\beta} \tr\Big\{P\ln P + (1- P)\ln(1 - P)\Big\}  \\
    & \text{s.t.} \ Q = P,\, R = P,\, \tr Q = N,\, 0 \preceq R \preceq
    I.
  \end{aligned}
\end{equation}
Note that we have explored the flexibility of the augmenting approach
to impose the trace constraint on $Q$, which is equivalent of imposing
the constraint on $P$. The minimization \eqref{eqn:LDM_0T_P_split} can be
iteratively solved by:
\begin{equation} \label{eqn:LDM_FT_PQR}
\begin{aligned}
  & (P^k,Q^k,R^k) = \arg\min_{P, Q, R \in \RR^{n\times n}} \frac{1}{\eta} \vertiii{Q}_1 + \tr(H P) + \frac{1}{\beta} \tr\Big\{P\ln P + (1- P)\ln(1 - P)\Big\} \\
  & \hspace{6em} +  \frac{\lambda}{2} \|P - Q + B^{k-1}\|_F^2 + \frac{r}{2}\|P - R + D^{k-1}\|^2_F  \\
  & \qquad  \text{s.t.} \qquad \tr Q = N,\, 0 \preceq R \preceq I,     \\
  & B^k = B^{k-1} + P^k - Q^k,   \\
  & D^k = D^{k-1} + P^k - R^k,
\end{aligned}
\end{equation}
where variables $B, D$ and parameters $r, \lambda$ have the similar roles as in \eqref{eqn:LDM_0T_PQR}.
Solving $P^k, Q^k, R^k$ in \eqref{eqn:LDM_FT_PQR} alternatively leads to the following three sub-optimization problems. 
\begin{enumerate}[$1^{\circ}$]
\item $\displaystyle P^k = \arg\min_{P \in \RR^{n\times n}}\tr(H P) +
  \frac{1}{\beta} \tr\Big\{P\ln P + (1- P)\ln(1 - P)\Big\} +
  \frac{\lambda}{2} \|P - Q^{k-1} + B^{k-1}\|_F^2 \\  + \frac{r}{2}\|P
  - R^{k-1} + D^{k-1}\|^2_F$.
\item $\displaystyle Q^k = \arg\min_{Q \in \RR^{n\times n}}
  \frac{1}{\eta}\vertiii{Q}_1+ \frac{\lambda}{2} \|P^k - Q +
  B^{k-1}\|_F^2, \quad \text{s.t.} \quad \tr Q = N $.
\item $R^k = \displaystyle\arg\min_{R  \in \RR^{n\times n}}  \|P^{k} - R + D^{k-1}\|^2_F, \quad \text{s.t.} \quad     0 \preceq R \preceq I$.
\end{enumerate}
Note that the sub-minimization problem $3^\circ$ of algorithm~\ref{alg:LDM_FT} can be solved explicitly, similarly as before 
\begin{equation*}
  R^k = V\min\{\max\{\Lambda,0\},1\}V^T , \text{ where } [V,~ \Lambda] =
  \operatorname{eig}(P^k + D^{k-1}).
\end{equation*}
To solve the sub-problem $2^\circ$ of algorithm~\ref{alg:LDM_FT}, let's denote by $Q_o$ the off-diagonal part of $Q$ and write $Q_d$ as the diagonal vector of $Q$. Namely, we have $Q = Q_o + \mathrm{diag}(Q_d)$. We also use similar notations for all other matrices. Then the solution of the sub-problem $2^\circ$ can be written as $Q^k = Q_o^k + \mathrm{diag}(Q_d^k)$, where $Q_o^k$ and $Q_d^k$ are given by (denoting $M^k = P^k + B^{k-1}$)
\begin{align}
  Q^k_o &= \operatorname{Shrink}\left(M^k_o,
    (\lambda\eta)^{-1} \right)
  = \operatorname{sign}(M^k_o  ) 
  \max\left\{ \abs{M^k_o  }  - (\lambda\eta)^{-1}, 0\right\}, \label{eqn:Q_o}\\
  Q^k_d &= \arg\min_{Q_d\in \RR^{n}}
  \frac{1}{\eta}\norm{Q_d}_1+ \frac{\lambda}{2} \| Q_d - M^k_d\|^2, \quad \text{s.t.} \quad \ {\bf 1}^TQ_d = N. \label{eqn:Q_d}
\end{align}
here, ${\bf 1}$ is a $n\times 1$ vector with all elements 1. Note that \eqref{eqn:Q_d} is a convex optimization problem with size $n$, which can be efficiently solved using the following Bregman iteration. 

\begin{equation}
\begin{aligned}
\displaystyle Q^{k,s}_d &= \operatorname{Shrink}\left(\frac{\lambda}{\lambda +r}M^k_d + \frac{r}{\lambda +r} (v^{s-1} - b^{s-1}),
    \frac{1}{(\lambda+r)\eta} \right). \\
v^s &= Q^{k,s}_d + b^{s-1} - \frac{1}{n}\left( {\bf 1}^T ( Q^{k,s}_d + b^{s-1}) - N\right). \\
b^{s} &= b^{s-1} +  Q^{k,s}_d - v^s.
\end{aligned}
\end{equation}


Next, we propose to iteratively solve the sub-minimization $1^\circ$ in algorithm~\ref{alg:LDM_FT}. 
\begin{multline}\label{eq:kkt}
\displaystyle  P^{k} =  \arg\min_{P \in \RR^{n\times n}}\tr(H P) + \frac{1}{\beta} \tr\Big\{P\ln P + (1- P)\ln(1 - P)\Big\} + \frac{\lambda}{2} \|P - Q^{k-1} + B^{k-1}\|_F^2 \\ 
+ \frac{r}{2}\|P - R^{k-1} + D^{k-1}\|^2_F.
\end{multline}
Note that this is a convex problem, and hence the existence of
uniqueness of $P^k$ is guaranteed. By the KKT condition, $P^k$ satisfies 
\begin{equation}
  H + \frac{1}{\beta} (\ln P^k - \ln (1 - P^k)) + \lambda(P^k - Q^{k-1} + B^{k-1}) + r ( P^k - R^{k-1} + D^{k-1}) = 0. 
\end{equation}
Equivalently, we may write the above equation as 
\begin{equation}
  P^k = \Bigl[ 1 + \exp \Bigl( \beta \bigl(
  H + \lambda(P^{k} - Q^{k-1} + B^{k-1}) + r ( P^{k} - R^{k-1} + D^{k-1}) \bigr) \Bigr) \Bigr]^{-1}. 
\end{equation}
Thus a natural iterative scheme to solve for $P^k$ is given by 
\begin{align}
  & Z^{l+1} = \frac{1}{1 + \exp(\beta Y^{l})}; \\
  & Y^{l} = H + \lambda(Z^l - Q^{k-1} + B^{k-1}) + r ( Z^l - R^{k-1} +
  D^{k-1}).
\end{align}
The following proposition gives the convergence of the above scheme.
\begin{prop}
\label{prop:fixedptconvg}
  Assume $\beta ( \lambda + r) < 4$. Given any symmetric matrices $Q, R, B, D\in\RR^{n\times n}$,
  the iteration scheme
  \begin{align*}
    & Z^{l+1} = \frac{1}{1 + \exp(\beta Y^{l})}; \\
    & Y^{l} = H + \lambda(Z^l - Q + B) + r ( Z^l - R + D) 
  \end{align*}
  converges exponentially as $l \to \infty$:
  \begin{equation}
    \norm{Z^l - Z^{\ast}} = \bigl(\beta ( \lambda + r) / 4 \bigr)^l 
    \norm{Z^0 - Z^{\ast}}, 
  \end{equation}
  where $Z^{\ast}$ is the unique solution of \eqref{eq:kkt}.
\end{prop}
\begin{proof}
  Denote the Fermi-Dirac function
  \begin{equation*}
    \phi_{\beta}(x) = \frac{1}{1 + \exp(\beta x)}.
  \end{equation*}
  We have then 
  \begin{equation*}
    \max_x\; \abs{\phi_{\beta}'(x)} = \max_x\; \abs{ \frac{\beta \exp(\beta x)}{\bigl(1 + \exp(\beta x)\bigr)^2} } \leq \beta / 4. 
  \end{equation*}
  Therefore, 
  \begin{equation*}
    \norm{Z^{l+1} - Z^{\ast}}  = \norm{ \phi_{\beta}(Y^l) - \phi_{\beta}(Y^{\ast})} 
    \leq \frac{\beta}{4} \norm{Y^l - Y^{\ast}}
    = \frac{\beta (\lambda + r)}{4} \norm{Z^{l} - Z^{\ast}}, 
  \end{equation*}
  where $Y^{\ast} = H + \lambda (Z^{\ast} - Q + B) + r (Z^{\ast} - R +
  D)$. The proposition then follows by iterating with respect to
  $l$. 
\end{proof}

In summary, we arrive at algorithm~\ref{alg:LDM_FT} for \eqref{eqn:LDM_FT}.
\begin{algorithm2e}
\caption{Finite temperature localized density matrix minimization.}
\label{alg:LDM_FT}
Initialize $Q^0 = R^0 = P^0 \in\mathcal{C} , B^0 = D^0 = 0$

\While{``not converge"}{
   \While{``not converge"}{ 
           $Y^{k,l} =  H + \lambda(Z^{k,l-1} - Q^{k-1} + B^{k-1}) + r ( Z^{k,l-1} - R^{k-1} + D^{k-1})$.
           
           $\displaystyle Z^{k,l} = \frac{1}{ 1 + \exp(\beta Y^{k,l})}$ .
          }
          
   $P^k = Z^{k,l}$.
   
 $\displaystyle Q^k  = Q_o^k + \mathrm{diag}(Q_d^k)$, where $Q_o^k$ and $Q_d^k$ are given by \eqref{eqn:Q_o} and \eqref{eqn:Q_d}.
  
  $R^k = V\min\{\max\{\Lambda,0\},1\}V^T , \text{ where } [V,~ \Lambda] =
  \operatorname{eig}(P^k + D^{k-1})$.
  
 $B^{k} = B^{k-1} +  P^k - Q^k$.
 
 $D^{k} = D^{k-1} +  P^k - R^k$.

}
\end{algorithm2e}

\begin{remark}
  In practice, it is not necessary to require all inner iterations
  convergence in algorithm \ref{alg:LDM_FT}. In our numerical
  experiments, we run all the inner iterations in algorithm
  \ref{alg:LDM_FT} within a given small number of steps, as $Z^{l}$
  converges exponentially by Proposition \ref{prop:fixedptconvg}.
\end{remark}

\begin{theorem}[Convergence of Algorithm~\ref{alg:LDM_0T} and Algorithm~\ref{alg:LDM_FT}]\leavevmode
\label{thm:LDM_ADM_Conv} 
  \begin{enumerate}
  \item The sequence $\big\{(P^k, Q^k, R^k)\big\}_k$ generated by
  algorithm~\ref{alg:LDM_0T} from any starting point converges to a
  minimum of the variational problem \eqref{eqn:LDM_0T}.
  \item The sequence $\big\{(P^k, Q^k, R^k)\big\}_k$ generated by
  algorithm~\ref{alg:LDM_FT} from any starting point converges to a
  minimum of the variational problem \eqref{eqn:LDM_FT}.
  \end{enumerate}
\end{theorem}
\begin{proof} The convergence of algorithm~\ref{alg:LDM_0T} is proved
  in~\cite{LaiLuOsher:15}. The proof in fact also applies verbatim to
  algorithm~\ref{alg:LDM_FT} as it is written for generic convex
  energy $E(P)$.
\end{proof}

While the minimization problems \eqref{eqn:LDM_0T} and
\eqref{eqn:LDM_FT} are convex and the proposed algorithms converge to
the minimizers by theorem~\ref{thm:LDM_ADM_Conv}, it is also clear
that the computational efficiency of the algorithms~\ref{alg:LDM_0T}
and \ref{alg:LDM_FT} is limited by the eigen-decomposition
in the step of eigenvalue thresholding, and also the inner iteration
in the finite temperature case. The computational cost for standard
eigen-decompsition algorithm is $\Or(n^3)$, which is rather expensive
for large scale calculations. In the next section, inspired by ideas
from linear scaling algorithms for electronic structure, we propose
approximate algorithms to solve \eqref{eqn:LDM_0T} and
\eqref{eqn:LDM_FT} by replacing eigendecompsition with polynomial
functions acting on matrices. The resulting algorithms have
computational cost linearly scaled with the matrix size $n$.

\section{Approximation by banded matrices and linear scaling algorithms}\label{sec:linear}

Based on results indicated in theorems~\ref{thm:LDM_0T_consist} and
\ref{thm:finitetemp}, the proposed LDM serves as a nice approximation
of the true density matrix.  More importantly, similar to many
$\ell_1$ regularization methods developed for compressed sensing
problems, the minimizing density matrices are expected to have certain
sparse structure.  More precisely, we denote the set of all banded
matrices with band width $w$ as
\begin{equation}
  \mathcal{B}_w = \left\{ P = (p_{ij}) \in \mathbb{R}^{n\times n} ~| P = P^T,   p_{ij} = 0,~ \forall   ~  j \notin \mathcal{N}^w_i \right \},
\end{equation}
where $\mathcal{N}^{w}_i$ denotes as a $w$-neighborhood of $i$. In
particular, for the $1D$ examples considered later in this paper, the
neighborhood is chosen as
\begin{equation*}
  \mathcal{N}_i^w = \{ j\in \{1, 2, \ldots, n\}  \mid \abs{i
    -j} \mod n  \leq w\},
\end{equation*}
for an given band width
$w \in \bigl\{0, 1, 2, \cdots, [\frac{n}{2}]\bigr\}$ ($w$ is typically
chosen much smaller than $n/2$). For example, consider a banded
discretized Hamiltonian $H$ (e.g., a central difference discretization
of Hamiltonian $-\frac{1}{2}\Delta +V$), numerical results
in~\cite{LaiLuOsher:15} suggest that the LDM is banded with a small
band width. We remark that, however, theoretical validation of this
observation is still open and remains to be investigated in future
works.

This  motivates the following variational problems to
approximate the LDMs proposed in \eqref{eqn:LDM_0T} and
\eqref{eqn:LDM_FT} by simply constraining the problems on the set of
banded matrices.
\begin{equation}\label{eqn:Band_LDM_0T}
  \begin{aligned}
    & \min_{P \in \RR^{n \times n}} \E^w_{\infty, \eta}(P) =  \min_{P \in \RR^{n\times n}} \tr(H P) + \frac{1}{\eta} \vertiii{P}_1, \\
    & \text{s.t.} \quad \tr P = N,\; P = P^{\TT}, \; 0 \preceq P \preceq I,  \; P\in\B_w.
  \end{aligned}
\end{equation}
\begin{equation}\label{eqn:Band_LDM_FT}
  \begin{aligned}
    & \min_{P \in \RR^{n \times n}} \E^w_{\beta, \eta}(P) 
    = \min_{P \in \RR^{n\times n}} \tr(H P) + \frac{1}{\beta} \tr\Big\{P\ln P + (1- P)\ln(1 - P)\Big\} + \frac{1}{\eta}\vertiii{P}_1\\
    & \hspace{2.5cm} \text{s.t.}  \quad  \tr P = N, \; P = P^{\TT}, \;  0 \preceq P \preceq I,  \; P\in\B_w.
  \end{aligned}
\end{equation}
Note that the above two optimization problems are still convex. This
is in contrast to the usual minimization problems developed in the
literature of linear scaling algorithms
\cite{McWeeny:60,Challacombe:99}.

The advantage of considering banded matrix is that it allows for
linear scaling algorithms. Let us consider first the eigenvalue
thresholding:
\begin{equation}\label{eq:RM}
  R =  V\min\{\max\{\Lambda,0\},1\}V^T , \text{ where } [V,~\Lambda] =
  \operatorname{eig}(M).
\end{equation}
Observe that the above formula can be written as a matrix function
\begin{equation}
  R = h(M), \qquad \text{where }  h(x) = \min\{ \max\{ x, 0\}, 1\} = 
  \min\{1/2(|x| +x),1\}.
\end{equation} 
The standard evaluation of the matrix function $h$ using spectral
theory needs diagonalization as in \eqref{eq:RM}. The key idea is to
approximate general matrix functions by polynomials, as the
polynomials acting on matrix only involves products and sums, which
can take advantage of the bandedness of the matrix. This type of
algorithms has been explored extensively in the literature of linear
scaling algorithms (see \cite{GoedeckerColombo:94} and the review articles
\cites{Goedecker:99, Niklasson:11, BowlerMiyazaki:12}).

More specifically in the current context, we will approximate the hard
thresholding function $h(x) = \min\{1/2(|x| +x),1\}$ using the
Chebyshev polynomial approximation 
\cite{Trefethen:13}. Recall that the standard Chebyshev polynomials
approximation minimizes the $L^{\infty}$ error on the interval
$[-1,~1]$, hence, we will first rescale the matrix such that its
eigenvalue lies in the interval.

For this, we first use the power method to estimate the largest
eigenvalue $\lambda_{\max}$ and the smallest eigenvalue
$\lambda_{\min}$ of $M$. An affine scaling of $M$ gives $s(M)$, whose
eigenvalue is in $[-1, 1]$, where
\begin{equation*}
  s(M) = \frac{2}{\lambda_{\max} - \lambda_{\min}} ( M - \lambda_{\min}) - 1. 
\end{equation*}
Note that its inverse is given by 
\begin{equation*}
  s^{-1}(M) =
  \frac{\lambda_{\max} - \lambda_{\min} }{2}  M + \frac{ \lambda_{\min} +
    \lambda_{\max} }{2}.
\end{equation*}
We approximate  
\begin{equation*}
  h(M) = h\circ s^{-1} ( s(M)) \approx T^m_{HT} (s(M)), 
\end{equation*}
where $T^m_{HT}$ is the Chebyshev polynomial approximation of $h\circ
s^{-1}$ on $[-1, 1]$.  Figure~\ref{fig:ChebApproximation}(a)
illustrates the Chebyshev polynomial approximation of $h$ (i.e.,
$h\circ s^{-1}$ assuming the scaling function $s$ is identity) by $40$
degree polynomials. Higher degree polynomial is needed if $M$ has a
larger spectrum span. In our numerical tests, it seems that fixing the
degree be $40$ gives satisfactory result for the test examples.

As acting $T^m_{HT}$ on $s(M)$ involves matrix products which will
increase the matrix band-width. The resulting matrix is projected
(with respect to Frobenius norm) on the space $\B_w$ to satisfy the
constraint. This projection is explicitly given by the truncation
operator $\T_{w} : \RR^{n\times n} \rightarrow \B_w$ which sets all
matrix entries to $0$ outside the $w$-band. 

To sum up, the eigenvalue thresholding can be approximated by
algorithm~\ref{alg:EigenHardThresholding}.
\begin{algorithm2e}[h]
  \caption{A linear scaling algorithm for eigenvalue thresholding.}
  \label{alg:EigenHardThresholding}
  \SetKwFunction{EigenThres}{EigenThresviaChebyPoly}%
  $\wh{P}~=~${\EigenThres{P,~m}}{
    \SetKwInOut{Input}{input}\SetKwInOut{Output}{output}
    \SetKwComment{Comment}{}{}
    
    \KwIn{$P\in\B_w, ~m$}
    \KwOut{$\wh{P}\in\B_w$}
    
    Estimate $\lambda_{\max}$ and $\lambda_{\min}$ of $P$ using the
    power method.
    
    Compute $T^m_{HT}(x)$ as the Chebyshev polynomial approximation of
    $h\left (\frac{\lambda_{\max} - \lambda_{\min} }{2}x + \frac{
        \lambda_{\min} + \lambda_{\max} }{2} \right )$.
    
    $\displaystyle P_{new} =  \frac{2 }{ \lambda_{\max} - \lambda_{\min}} (P - \lambda_{\min} )- 1$.
    
    $\wh{P} = \displaystyle \T_{w} T^m_{HT}(P_{new})$.
  }
\end{algorithm2e}
For banded matrix with band width $w$, the computational cost of this
algorithm is $O(nm^2w)$ which scale linearly with respect to the
matrix dimension $n$.  Therefore, for matrices in $\B_w$, this
algorithm has much better computational efficiency in particular for
matrices of large size.

Replacing the eigenvalue thresholding in
algorithm~\ref{alg:LDM_0T_linear} by
algorithm~\ref{alg:EigenHardThresholding}, we obtain the following
algorithm for computing LDM in the zero temperature case.
\begin{algorithm2e}[h]
  \caption{A linear scaling algorithm for computing LDM at zero
    temperature.}
  \label{alg:LDM_0T_linear}

  Initialize $Q^0 = R^0 = P^0 \in\C \bigcap \B_w , B^0 = D^0 = 0$,
  choose $m_{et}$.

  \While{``not converge"}{
    $\displaystyle P^k = \Gamma^k - \frac{\tr (\Gamma^k) - N}{n},
    \quad \text{where } \Gamma^k = \frac{\lambda}{\lambda+r}(Q^{k-1} -
    B^{k-1}) + \frac{r}{\lambda+r}(R^{k-1} - D^{k-1}) -
    \frac{1}{\lambda +r} H $.
  
    $\displaystyle Q^k = \operatorname{Shrink}\left(P^k
      +B^{k-1},\frac{1}{\lambda\eta} \right)$.

    $R^k = \textrm{\sf EigenThresviaChebyPoly}(P^k + D^{k-1},m_{et})$.
 
    $B^{k} = B^{k-1} + P^k - Q^k$.
 
    $D^{k} = D^{k-1} + P^k - R^k$.  }
\end{algorithm2e}
We remark that all the steps in algorithm~\ref{alg:LDM_0T_linear}
preserves the bandedness of the matrices: For the approximate
eigenvalue thresholding, this is guaranteed by the explicit projection
step; it can be easily checked for the other steps. Hence, the
iterates of the algorithm $(Q, R, P, B, D)$ will remain in $\B_w$ as
the initial condition lies in the set. Thus, the whole algorithm is
linear scaling.

\smallskip 

For the finite temperature case, we need a further approximation for
evaluating the Fermi-Dirac matrix function $(1 + \exp(\beta Y))^{-1}$,
as direct computation also involves matrix diagonalization, which is
$\Or(n^3)$. Following the same procedure as in
algorithm~\ref{alg:EigenHardThresholding}, we can achieve a linear
scaling algorithm by a Chebyshev polynomial approximation of the
Femi-Dirac function $\phi_{\beta}(x) = (1 + \exp(\beta x))^{-1}$. 

This leads to an algorithm $\textrm{\sf FermiDiracviaChebyPoly}(Y,m)$
for approximating the Fermi-Dirac operation by simply replacing the
hard thresholding function $h(x)$ with the Fermi-Dirac function
$\phi_\beta(x)$ in algorithm \ref{alg:EigenHardThresholding}. Hence,
we omit the details. Figure~\ref{fig:ChebApproximation}(b) shows the
approximation of $\phi_\beta(x)$ for $\beta = 10$ with a degree $40$
polynomial. Nice agreement is observed. We remark that if the
temperature is lower (so that $\beta$ is larger), a higher degree
polynomial is needed as the function $\phi_{\beta}(x)$ has larger
derivatives. In fact, as $\beta \to \infty$, $\phi_{\beta}$ converges
to a Heaviside function with jump at $x = 0$.
Thus, we arrive at the following algorithm~\ref{alg:LDM_FT_Linear} for
computing LDM of the finite temperature case with linear scaling
computation cost with respect to the matrix size $n$.
\begin{figure}[htp]
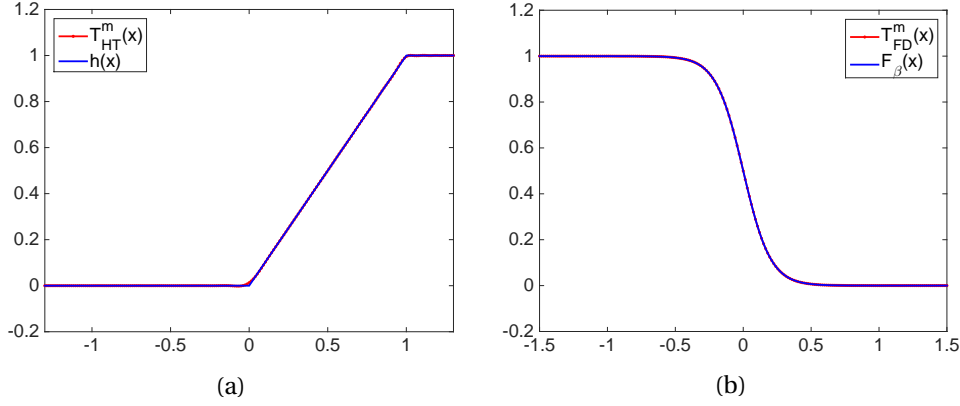

\label{fig:ChebApproximation}
\begin{center}
\begin{minipage}{0.47\linewidth}
\includegraphics[width=1\linewidth]{SoftThres_Cheby.eps}
\centering (a)
\end{minipage}\hfill
\begin{minipage}{0.48\linewidth}
\includegraphics[width=1\linewidth]{FemiDirac_Cheby.eps}\\
\centering (b)
\end{minipage}\hfill
\end{center}
\caption{(a): Chebyshev approximation of $h(x)$ with $m = 50$ degree
  polynomial. (b): Chebyshev approximation of $\phi_{\beta}(x)$ for
  $\beta = 10$ with $m = 40$ degree polynomial.}
\end{figure}

%
%
%
%
%
%
\begin{algorithm2e}
  \caption{A linear scaling algorithm for computing LDM at finite
    temperature.}
\label{alg:LDM_FT_Linear}
Initialize $Q^0 = R^0 = P^0 \in\C \bigcap \B_w , B^0 = D^0 = 0$, choose $m_{fd},m_{et}$

\While{``not converge"}{
   \While{``not converge"}{ 
           $Y^{k,l} =  H + \lambda(Z^{k,l-1} - Q^{k-1} + B^{k-1}) + r ( Z^{k,l-1} - R^{k-1} + D^{k-1})$.
           
           $\displaystyle Z^{k,l} = \textrm{\sf FermiDiracviaChebyPoly}(Y^{k,l},m_{fd})$ .
          }
          
   $P^k = Z^{k,l}$.
   
 $\displaystyle Q^k  = Q_o^k + \mathrm{diag}(Q_d^k)$, where $Q_o^k$ and $Q_d^k$ are given by \eqref{eqn:Q_o} and \eqref{eqn:Q_d}.
   
  $R^k =\textrm{\sf EigenThresviaChebyPoly}(P^k + D^{k-1}, m_{et})$.
  
 $b^{k} = B^{k-1} +  P^k - Q^k$.
 
 $d^{k} = D^{k-1} +  P^k - R^k$.

}
\end{algorithm2e}

We remark that since polynomial products are applied to approximate
the hard thresholding function and the Femi-Dirac operation,
approximation errors have been introduced in each iteration of the
proposed algorithms~\ref{alg:LDM_0T_linear} and \ref{alg:LDM_FT_Linear}. Therefore, the convergence proof
used in theorem~\ref{thm:LDM_ADM_Conv} as we discussed in
~\cite{LaiLuOsher:15} can not be directly applied, although our
numerical results in Section~\ref{sec:experiments} illustrate
satisfactory approximation to model \ref{eqn:DM_0T} and model
\ref{eqn:DM_FT}. Note that similar issues arise in theoretical
understanding of convergence of other iterative linear scaling
algorithms, for example \cite{GarciaLuXuanE:09}.

\section{Numerical Experiments}
\label{sec:experiments}
In this section, numerical experiments are presented to demonstrate
the proposed models and algorithms for LDM computing for at zero and
finite temperatures. We conduct numerical comparisons of our results
between LDM computation with and without the linear scaling
algorithms, which indicates satisfactory results of the proposed
linear scaling algorithms based on approximation by banded
matrices. We further illustrate efficiency of the proposed linear
scaling algorithms. All numerical experiments are implemented by
\textsf{MATLAB} in a PC with a 16G RAM and a 2.7 GHz CPU.

\begin{figure}[ht]
\centering
\begin{minipage}{0.6\linewidth}
\includegraphics[width=1\linewidth]{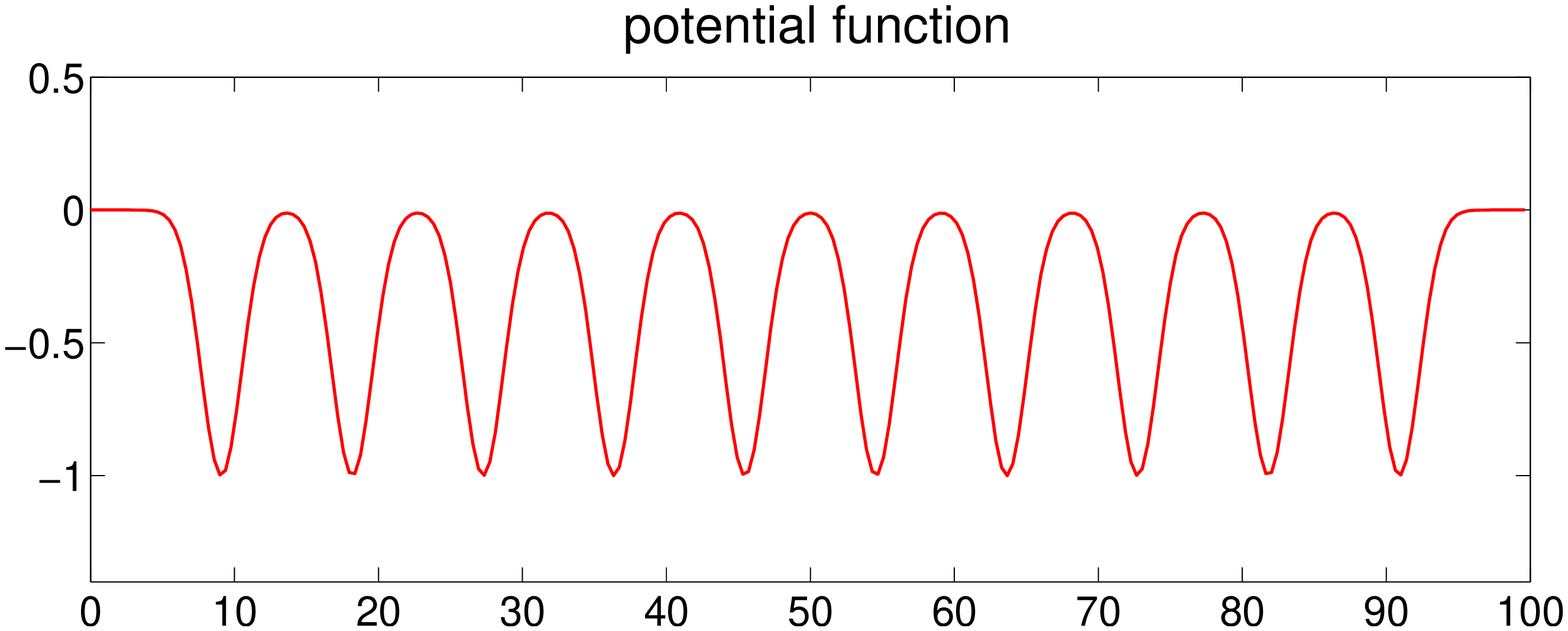}\\
\end{minipage}\hfill
\begin{minipage}{0.39\linewidth}
\includegraphics[width=.9\linewidth]{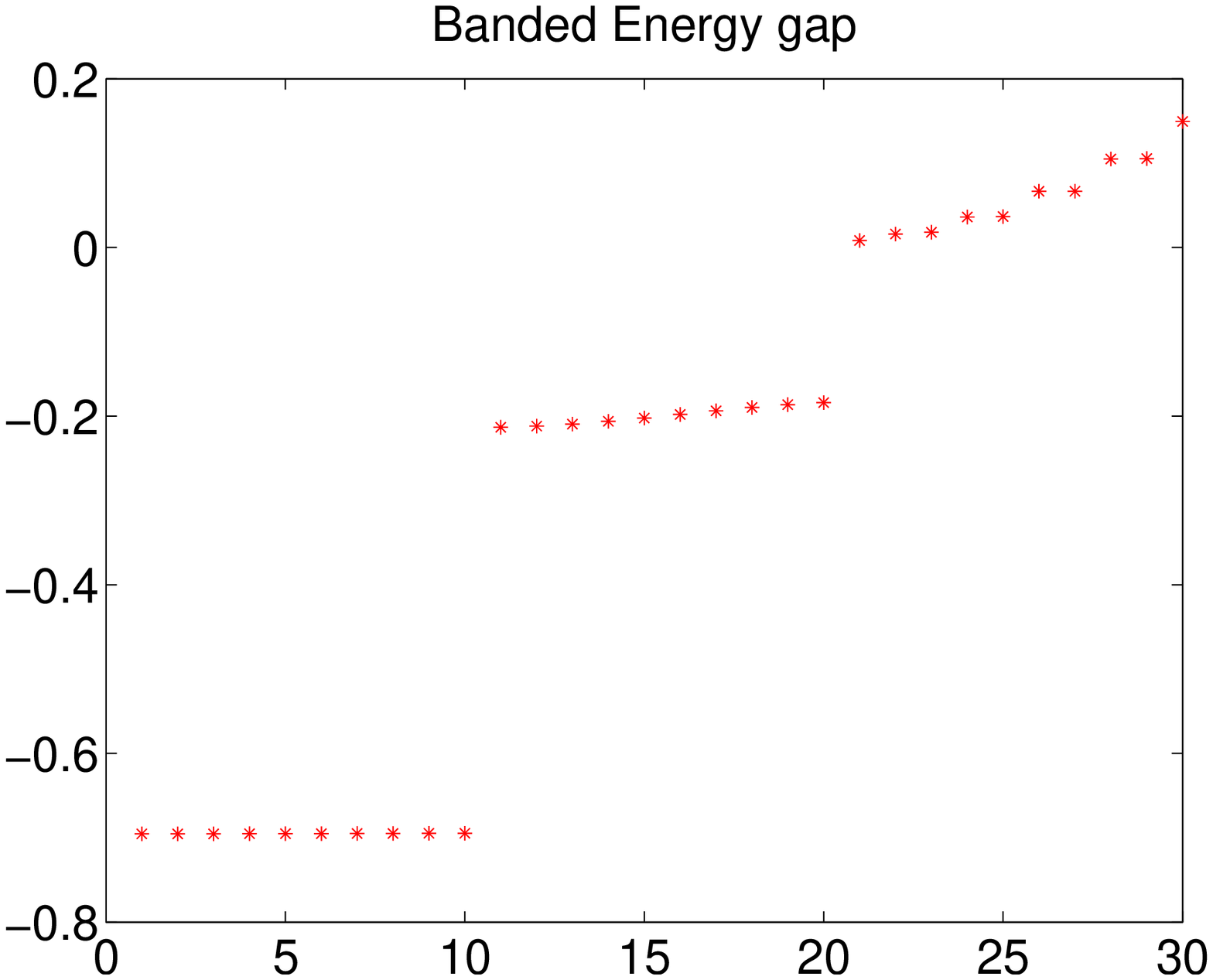}\\
\end{minipage}\hfill\\
\begin{minipage}{0.6\linewidth}
\centering (a)
\end{minipage}\hfill
\begin{minipage}{0.39\linewidth}
\centering (b)
\end{minipage}
\caption{(a): The potential function in the modified Kronig-Penney
  model. (b): The spectrum of the (discretized) Hamiltonian operator.}
 \label{fig:V_KP}
\end{figure}

In our experiments, we consider the proposed models defined on 1D domain $\Omega =
[0,~ 100]$ with periodic boundary condition. In this case, the matrix $H$ is a discretization of the
Schr\"{o}dinger operator $\frac{1}{2}\Delta + V$ defined on $\Omega$. Here, $\Delta$ is approximated by a central difference on $[0, ~100]$ with equally spaced $400$ points. In addition, we consider a modified Kronig--Penney (KP)
model~\cite{Kronig:1931quantum} for a one-dimensional insulator.  The
original KP model describes the states of independent electrons in a
one-dimensional crystal, where the potential function $V(x)$ consists
of a periodic array of rectangular potential wells. We replace the
rectangular wells with inverted Gaussians so that the potential is
given by 
\begin{equation}
\label{eqn:V_KP}
  V(x) = -V_0\sum_{j=1}^{N_{\text{at}}} \exp\left[ -\frac{(x -
      x_j)^2}{\delta^2} \right], 
\end{equation}
where $N_{\text{at}}$ gives the number of potential wells. In our
numerical experiments, we choose $N_{\text{at}} = 10$ and $x_j = 100 j
/ 11$ for $j = 1, \ldots, N_{\text{at}}$. The potential is plotted in
Figure~\ref{fig:V_KP}(a). For this given potential, the Hamiltonian
operator $H = - \tfrac{1}{2} \Delta + V(x)$ exhibits two low-energy
bands separated by finite gaps from the rest of the eigenvalue
spectrum (See Figure~\ref{fig:V_KP}(b)). 

The first experiment demonstrates the proposed methods at zero
temperature for an insulating case, where we set parameters as $\eta =
100, N = 10, m_{et} = 50$. Figure~\ref{fig:0TConvexSparseComp}(a)
illustrates the true density matrix $\sum_{1\leq i \leq 10} \phi_i
\phi_i^T$ obtained by the first $10$ eigenfunctions $\{ \phi_i
\}_{i=1}^{10}$ of $H$.  Figure~\ref{fig:0TConvexSparseComp}(b) plots
the density matrix obtained from the proposed model using algorithm
\ref{alg:LDM_0T}. As one can observe from
Figure~\ref{fig:0TConvexSparseComp}(b), the LDM provides a quite good
approximation to the true density matrix. Note that the LDM with a
larger $\eta$ imposes a smaller penalty on the sparsity, and hence the
solution has a wider spread and closer to the true density matrix. The
approximation behavior is stated in theorem
\ref{thm:LDM_0T_consist}. We also refer our previous
work~\cite{LaiLuOsher:15} for more numerical discussion about this
point. Figure~\ref{fig:0TConvexSparseComp}(c, d) illustrate computing
results of LDM using algorithm \ref{alg:LDM_0T_linear} with band size
$w = 10, 20$ respectively. As we can see from the results, a moderate
size of band width can provide satisfactory approximation for the
LDM. In addition, we also conduct quantitative comparisons between
algorithm \ref{alg:LDM_0T} and algorithm \ref{alg:LDM_0T_linear} in
table \ref{tab:LDM_Accuracy}, where $P$ and $P_w$ are results obtained
from algorithm \ref{alg:LDM_0T} and algorithm \ref{alg:LDM_0T_linear}
respectively. The relative energy error $\frac{|\tr(HP) -
  \tr(HP_w)|}{|\tr(HP)|}$ and the comparable relative differences
between $\frac{\|P - P_w\|_F}{\|P\|_F}$ and $ \frac{\| \T_w(P) -
  P_w\|_F}{\| P\|_F}$ indicate that results obtained by linear scaling
algorithm \ref{alg:LDM_0T_linear} can provide good estimation of the
LDM created from algorithm \ref{alg:LDM_0T}.

\begin{figure}[h]
\begin{center}
\begin{minipage}{0.49\linewidth}
\includegraphics[width=1\linewidth]{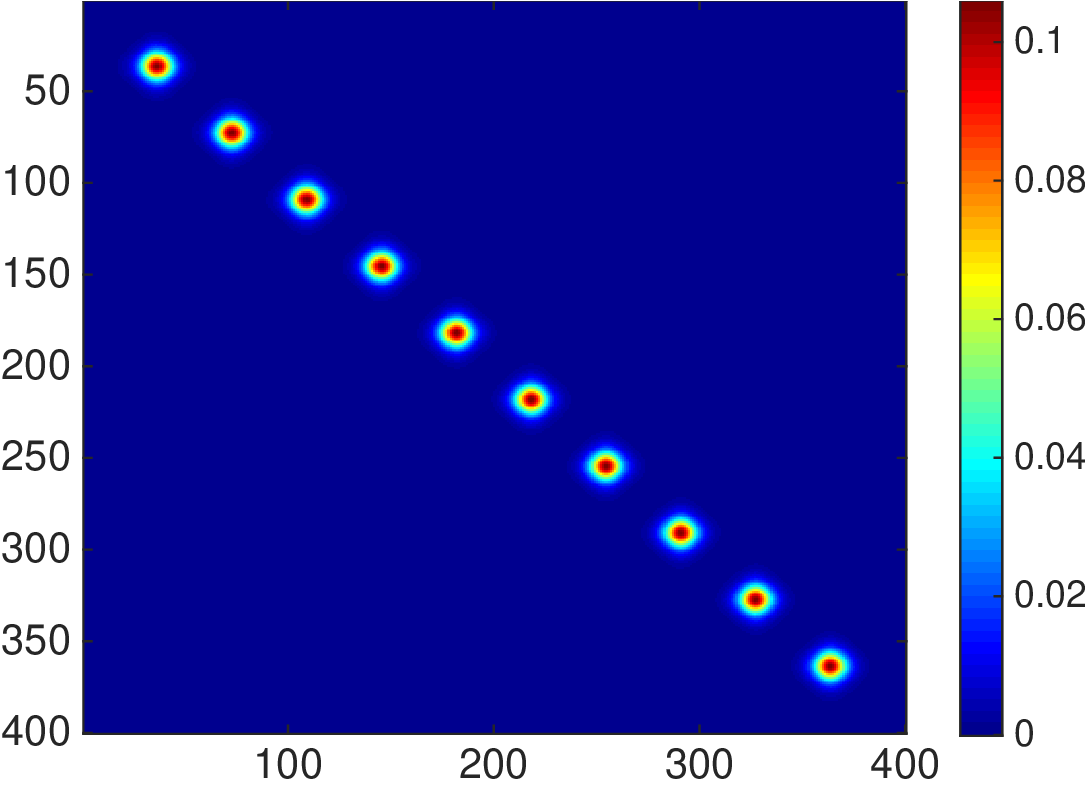}\\
\centering (a)
\end{minipage}\hfill
\begin{minipage}{0.49\linewidth}
\includegraphics[width=1\linewidth]{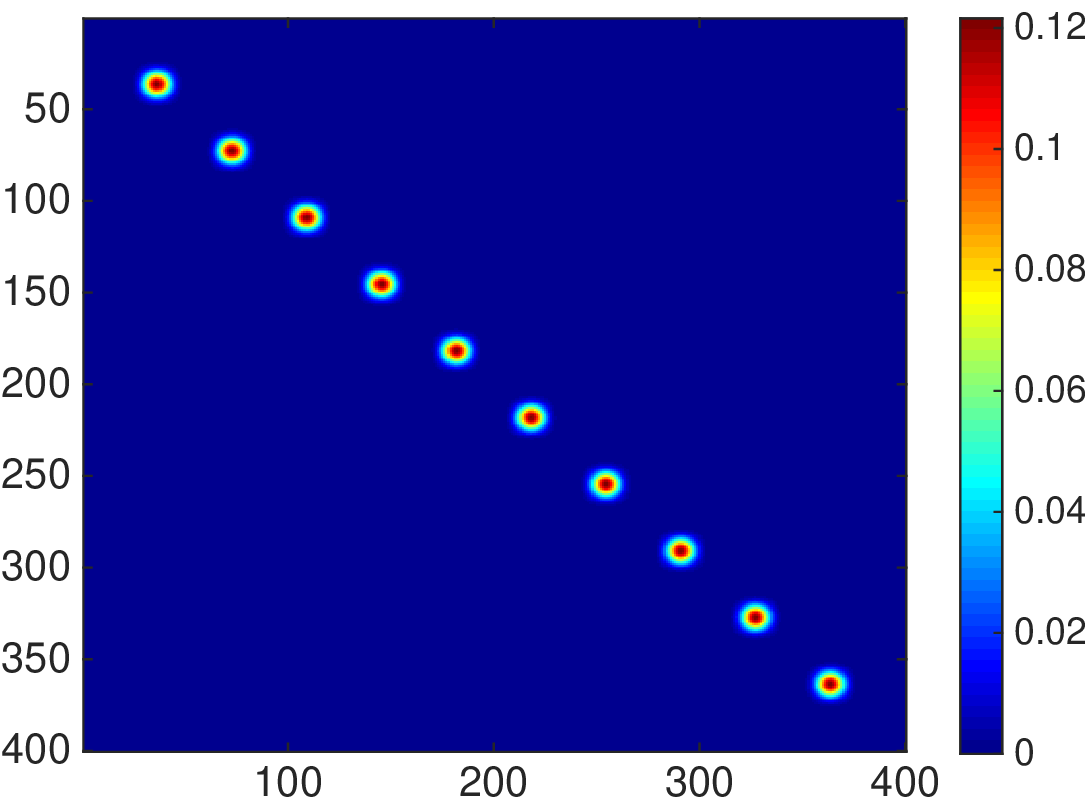}\\
\centering (b)
\end{minipage}\hfill\\
\begin{minipage}{0.49\linewidth}
\includegraphics[width=1\linewidth]{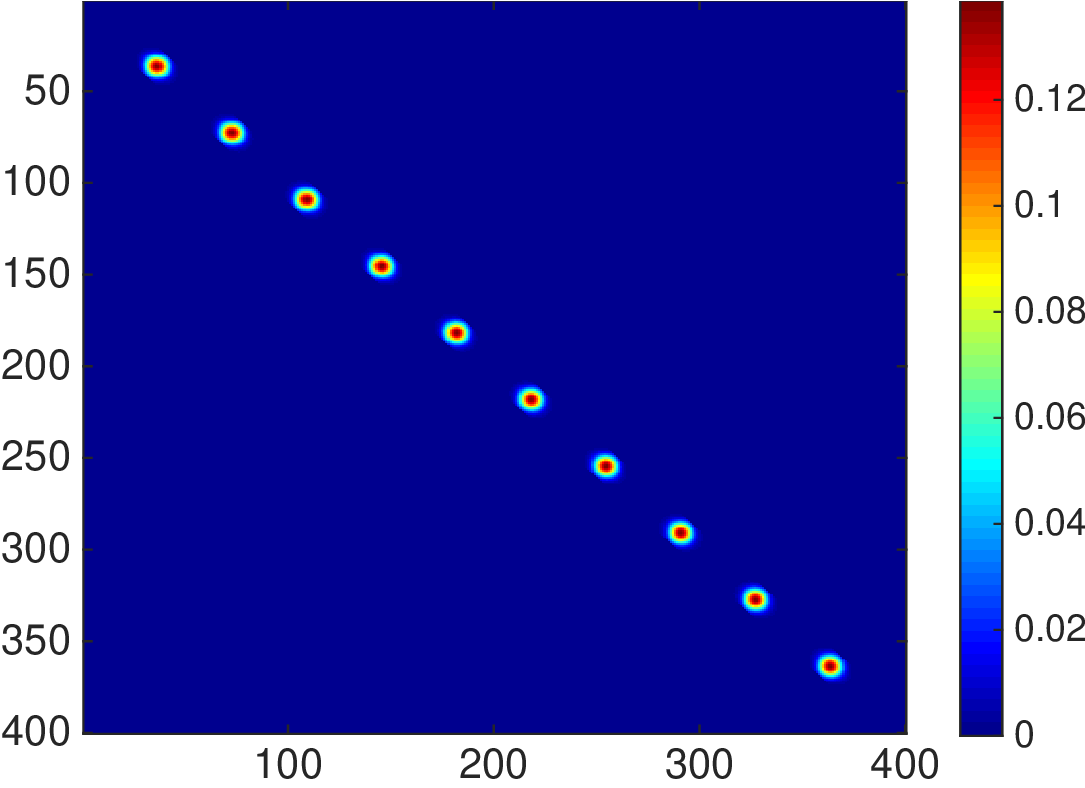}\\
\centering (c)
\end{minipage}\hfill
\begin{minipage}{0.49\linewidth}
\includegraphics[width=1\linewidth]{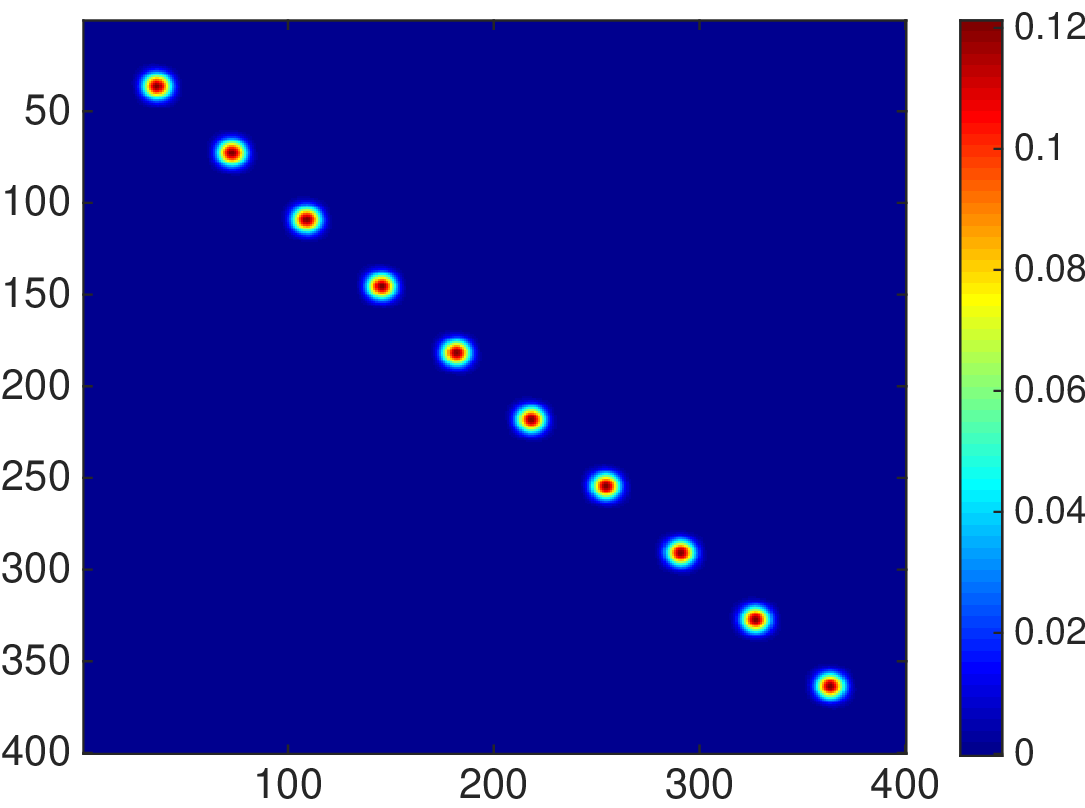}\\
\centering (d)
\end{minipage}\hfill\\
\end{center}
\caption{Result comparisons of the zero temperature case with KP potential ($n = 400, N = 10$). (a). The true density function obtained by the first N eigenfunctions of $H$. (b). Solution of the localized density matrix with $\eta  = 100$ using algorithm \ref{alg:LDM_0T}. (c), (d). Solutions of the localized density matrices ($\eta = 100$) using algorithm \ref{alg:LDM_0T_linear} with $w = 10, 20$ respectively. }
\label{fig:0TConvexSparseComp}
\end{figure}

For the finite temperature case, it is known that the true density matrix decays exponentially fast along the off-diagonal direction. In the second experiment, we test the algorithms for the finite temperature model with potential free case and modified KP case. We set parameters $\eta = 100, \beta = 1, m_{et} = 40, m_{fd} = 20$ for both cases.
Figure~\ref{fig:FTConvexSparseComp_KP}(a) and Figure~\ref{fig:FTConvexSparseComp_FreeV}(a) illustrate the corresponding true density
matrices described by $\sum_{i} \rho_i \phi_i  \phi_i^T$ using \eqref{eq:Pbetainf}, whose non-zero entries concentrate on a narrow band along the diagonal direction.  Figure~\ref{fig:FTConvexSparseComp_KP}(b) and Figure~\ref{fig:FTConvexSparseComp_FreeV}(b) plot the density matrix obtained from the proposed model using algorithm \ref{alg:LDM_FT}. It is clear that the LDM performs reasonably good approximation to the true density matrix. Figure~\ref{fig:FTConvexSparseComp_KP}(c, d) and Figure~\ref{fig:FTConvexSparseComp_FreeV}(c, d) illustrate computing results of LDM using algorithm \ref{alg:LDM_FT_Linear} with band size $w = 10, 20$ respectively. As we can see from the results, a moderate size of band width can provide satisfactory approximation for the LDM as the true density matrix has exponential decay property. 
In addition, we also conduct quantitative comparisons between algorithm \ref{alg:LDM_FT} and algorithm \ref{alg:LDM_FT_Linear} in table \ref{tab:LDM_Accuracy}. The relative energy error $\frac{|\tr(HP) - \tr(HP_w)|}{|\tr(HP)|}$ and the comparable relative differences between $\frac{\|P - P_w\|_F}{\|P\|_F}$ and $ \frac{\| \T_w(P) - P_w\|_F}{\| P\|_F}$ indicate that results obtained by linear scaling algorithm \ref{alg:LDM_FT_Linear} can provide satifactory estimation of the LDM created from algorithm \ref{alg:LDM_FT}.

\begin{figure}[h]
\begin{center}
\begin{minipage}{0.49\linewidth}
\includegraphics[width=1\linewidth]{FT_DM_n400_N10_beta1.eps}\\
\centering (a)
\end{minipage}\hfill
\begin{minipage}{0.49\linewidth}
\includegraphics[width=1\linewidth]{FT_LDM_n400_N10_beta1_mu100.eps}\\
\centering (b)
\end{minipage}\hfill\\
\begin{minipage}{0.49\linewidth}
\includegraphics[width=1\linewidth]{FT_BLDM_n400_N10_beta1_mu100_B10_Ceigthre40_chebfemi20.eps}\\
\centering (c)
\end{minipage}\hfill
\begin{minipage}{0.49\linewidth}
\includegraphics[width=1\linewidth]{FT_BLDM_n400_N10_beta1_mu100_B20_Ceigthre40_chebfemi20.eps}\\
\centering (d)
\end{minipage}\hfill\\
\end{center}
\caption{Result comparisons of the finite temperature case with KP potential ($n = 400, N = 10, \beta = 1$). (a). The true density function obtained by the first N eigenfunctions of $H$. (b). Solution of the localized density matrix with $\eta  = 100$ using algorithm \ref{alg:LDM_FT}. (c), (d). Solutions of the localized density matrices ($\eta = 100$) using algorithm \ref{alg:LDM_FT_Linear} with $w = 10, 20$ respectively. }
\label{fig:FTConvexSparseComp_KP}
\end{figure}

\begin{figure}[h]
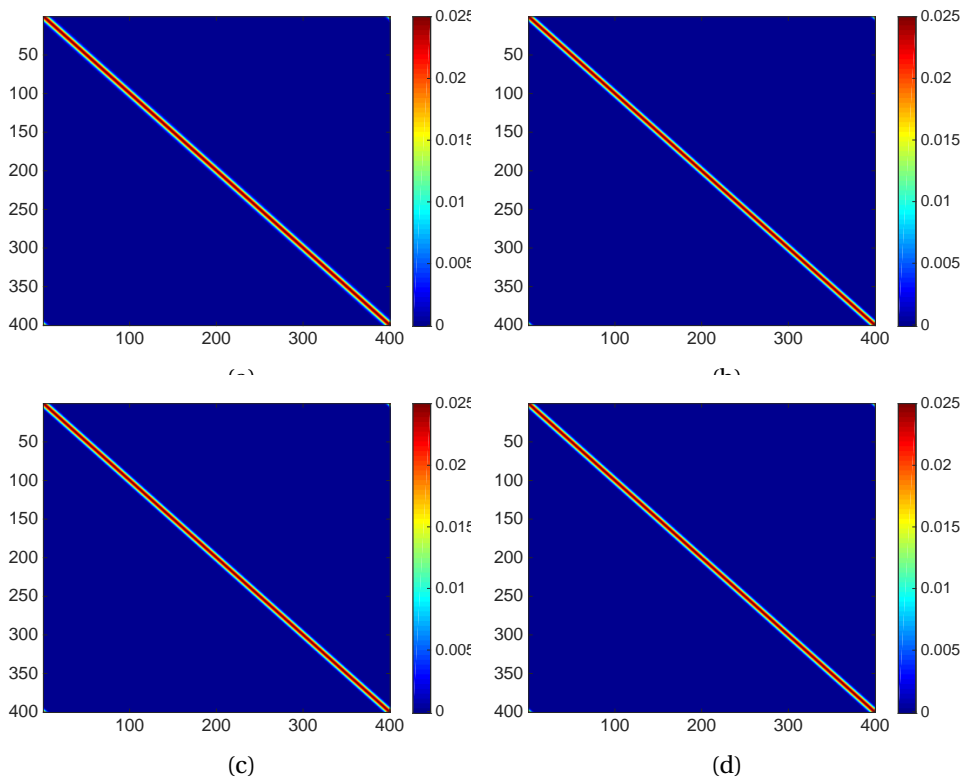

\begin{center}
\begin{minipage}{0.49\linewidth}
\includegraphics[width=1\linewidth]{FT_DM_n400_N10_beta1_freeV.eps}\\
\centering (a)
\end{minipage}\hfill
\begin{minipage}{0.49\linewidth}
\includegraphics[width=1\linewidth]{FT_LDM_n400_N10_beta1_mu100_freeV.eps}\\
\centering (b)
\end{minipage}\hfill\\
\begin{minipage}{0.49\linewidth}
\includegraphics[width=1\linewidth]{FT_BLDM_n400_N10_beta1_mu100_B10_Ceigthre40_chebfemi20_freeV.eps}\\
\centering (c)
\end{minipage}\hfill
\begin{minipage}{0.49\linewidth}
\includegraphics[width=1\linewidth]{FT_BLDM_n400_N10_beta1_mu100_B20_Ceigthre40_chebfemi20_freeV.eps}\\
\centering (d)
\end{minipage}\hfill\\
\end{center}
\caption{Result comparisons of the potential free finite temperature case($n = 400, N = 10, \beta = 1$). (a). The true density function obtained by the first N eigenfunctions of $H$. (b). Solution of the localized density matrix with $\eta  = 100$ using algorithm \ref{alg:LDM_FT}. (c), (d). Solutions of the localized density matrices ($\eta = 100$) using algorithm \ref{alg:LDM_FT_Linear} with $w = 10, 20$ respectively. }
\label{fig:FTConvexSparseComp_FreeV}
\end{figure}

\begin{table}[h]
\begin{center}
\begin{tabular}{|c||c||c|c||c|c|}
  \hline
  &  & $ \frac{\abs{\tr(HP) - \tr(HP_w)}}{\abs{\tr(HP)}}$  & $\frac{| \E_{\beta,\eta}(P) - \E_{\beta,\eta}(P_w)|}{|\E_{\beta,\eta}(P)|} $    & $\frac{\| P - \T_w(P) \|_F}{\|P\|_F} $&  $\frac{\| P - P_{w} \|_F}{\|P\|_F} $ \\ 
  \hline
  \multirow{3}{3cm}{zero temperature ($\beta =\infty, \eta = 100$, modified KP)} & $w = 10$ & $1.00\times 10^{-1}$ & $6.36 \times 10^{-2}$ &$1.46\times 10^{-1}$ & $2.32 \times 10^{-1}$ \\
  \cline{2-6}
  & $w = 15$ &  $4.30\times 10^{-2}$ & $5.28 \times 10^{-2}$ &$2.24 \times 10^{-2}$    &  $3.27\times 10^{-2} $\\ 
  \cline{2-6}
  & $w = 20$ & $4.97\times 10^{-3}$  &$1.20 \times 10^{-2}$ &$1.01\times 10^{-3} $  & $1.56 \times 10^{-2}$\\  
  \hline\hline
  \multirow{3}{3cm}{finite temperature ($\beta = 1, \eta = 100$, modified KP)} & w = 10 & $8.31 \times 10^{-3}$ & $1.44 \times 10^{-2}$&$3.86\times 10^{-3}$ & $8.46 \times 10^{-3}$ \\
  \cline{2-6}
  & $w = 15$ & $5.87 \times 10^{-4}$ &  $5.78 \times 10^{-4}$ &$2.51\times 10^{-4}$  &  $5.38\times 10^{-4} $ \\
  \cline{2-6}
  & $w = 20$ & $9.40 \times 10^{-4}$ & $4.12 \times 10^{-4}$ &$1.36\times 10^{-4} $  & $3.05 \times 10^{-4}$ \\ 
  \hline\hline
  \multirow{3}{3cm}{finite temperature ($\beta = 1,  \eta = 100$, potential free)} & w = 10 & $2.13 \times 10^{-3}$ & $2.78 \times 10^{-2}$ &$4.04\times 10^{-3}$ & $5.99 \times 10^{-3}$ \\
  \cline{2-6}
  & $w = 15$ & $3.03 \times 10^{-6}$ &  $1.03 \times 10^{-3}$ & $2.66\times 10^{-4}$    &  $4.81 \times 10^{-4} $ \\
  \cline{2-6}
  & $w = 20$ & $1.42 \times 10^{-5}$ & $6.82 \times 10^{-4}$ &$1.54\times 10^{-4} $  & $2.35 \times 10^{-4}$ \\ 
  \hline
\end{tabular}
\end{center}
\caption{Approximation error using banded matrices with $\eta = 100$ for all cases. }
\label{tab:LDM_Accuracy}
\end{table}


\begin{figure}[h]
\begin{center}
\begin{minipage}{0.5\linewidth}
\includegraphics[width=1\linewidth]{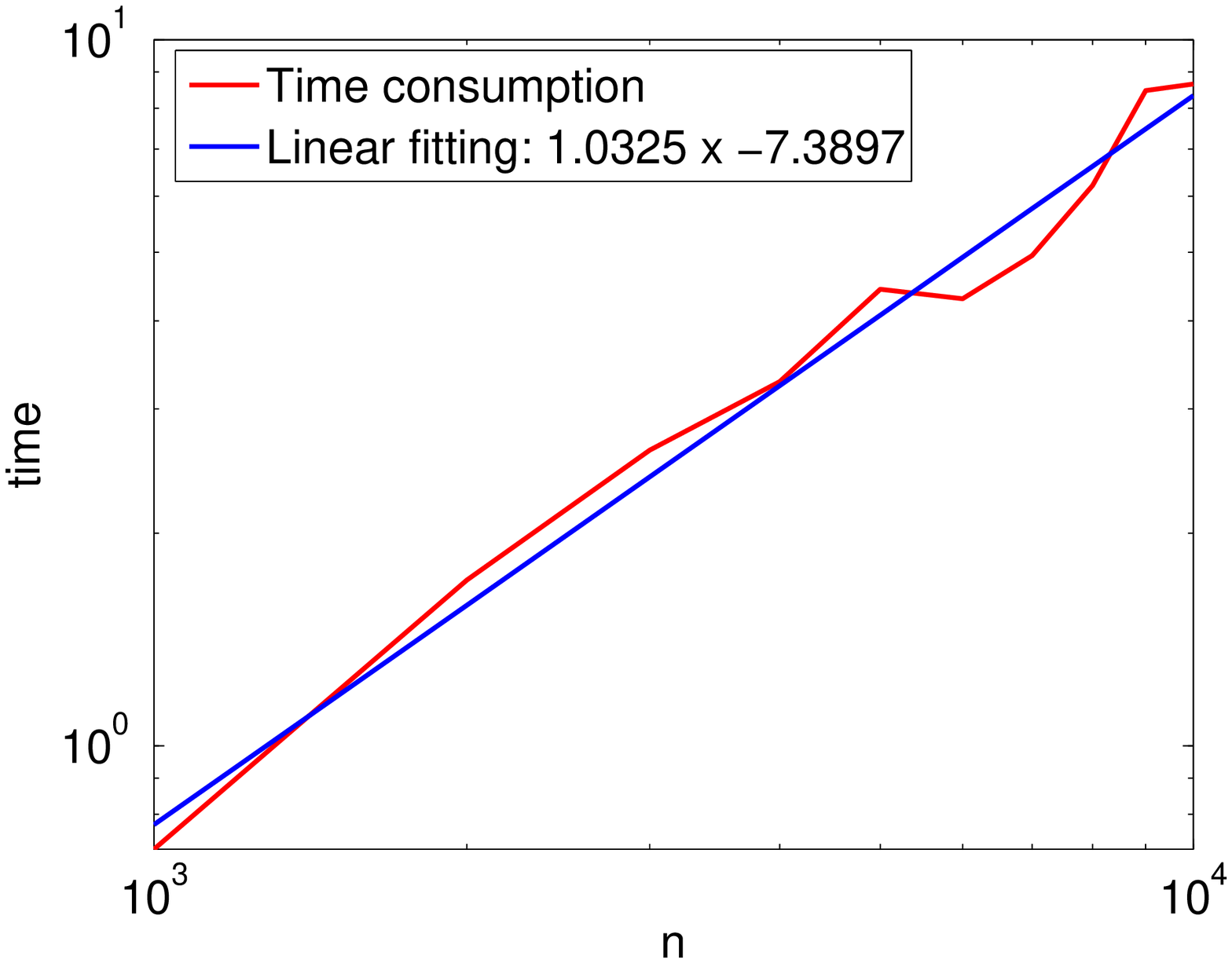}
\centering  (a)
\end{minipage}\hfill
\begin{minipage}{0.5\linewidth}
\includegraphics[width=1\linewidth]{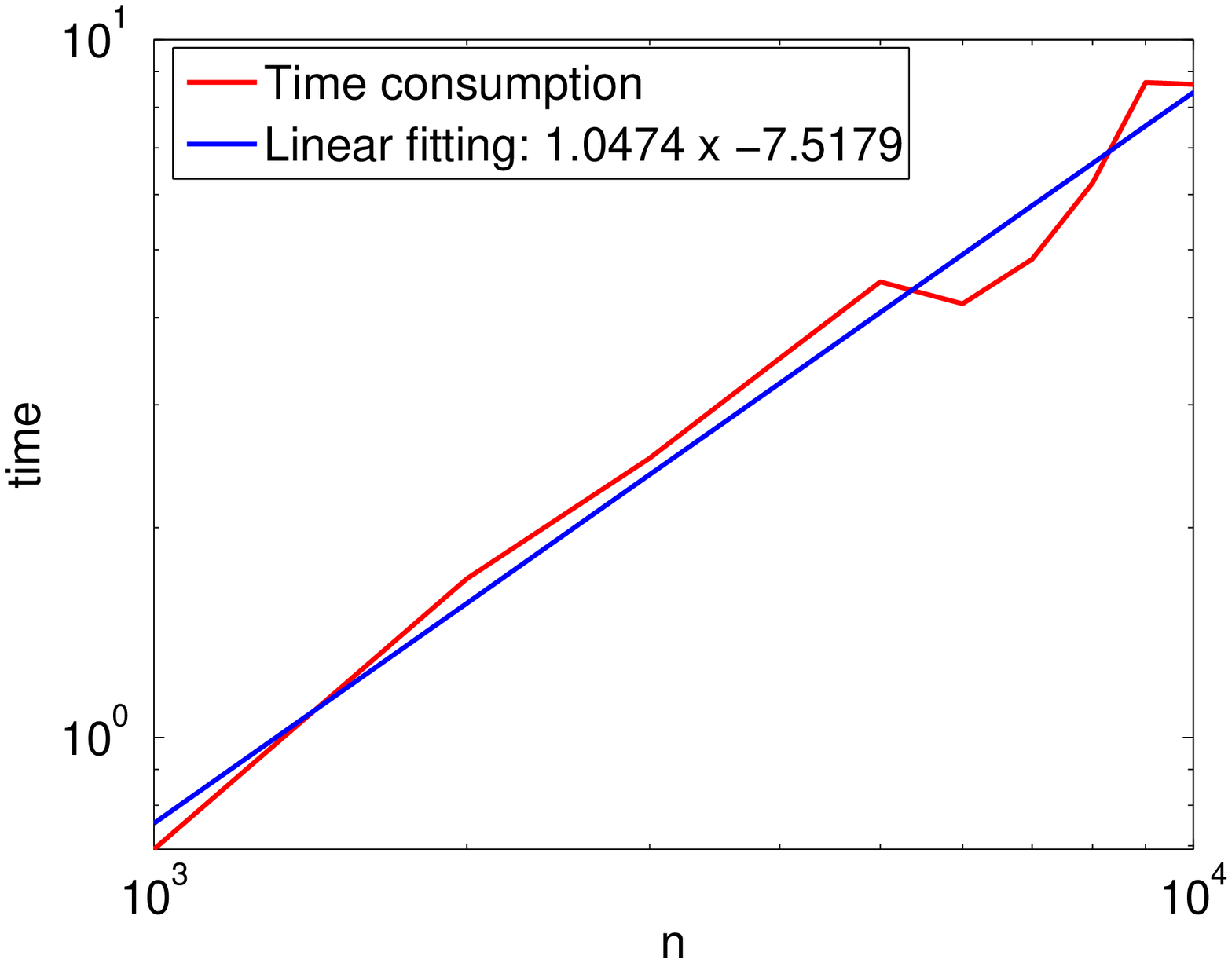}
\centering  (b)
\end{minipage}\hfill
\end{center}
\caption{(a): Time consumption of $\textrm{\sf EigenThresviaChebyPoly}(P,40)$ with different sizes of matrix $P$ and fixed band width $w = 10$. (b): Time consumption of $\textrm{\sf FermiDiracviaChebyPoly}(P,40)$ with different sizes of matrix $P$ and fixed band width $w = 10$.}
\label{fig:ChebApprox_LinearScaling}
\end{figure}

In the third experiment, we test the computation cost of the proposed algorithms for approximating the eigenvalue thresholding and the Fermi-Dirac operation. Using different matrix sizes, the $\textrm{\sf EigenThresviaChebyPoly}$ designed by algorithm \ref{alg:EigenHardThresholding} is applied to approximate the step of hard thresholding eigenvalues used in the algorithm \ref{alg:LDM_0T} step 5. Similarly, we also test $\textrm{\sf FermiDiracviaChebyPoly}$ for different matrix sizes for the Fermi-Dirac operation. Red curves in Figure \ref{fig:ChebApprox_LinearScaling} (a, b) report log-log curves of the average computation cost for both cases with respect to different matrix sizes, where blue curves illustrate the corresponding linear fitting for the computation cost curve. For banded matrices, It is clear that the computation costs for both approximation algorithms based on Chebyshev polynomials are linear scaling to the matrix sizes. 

We next report comparisons for the computation cost of all proposed algorithms, where we set parameters the same as those used in the first and the second experiments. We first conduct comparison between algorithm \ref{alg:LDM_0T} and algorithm \ref{alg:LDM_0T_linear} for the zero temperature case with matrix size chosen from $10^3$ to $10^4$. Figure \ref{fig:LinearScaling} (a) reports the log-log plots of average time consumption of each iteration for both algorithms. From the slope of the corresponding linear fitting curves, it is clear to see that the original proposed algorithm \ref{alg:LDM_0T} has computation cost cubically dependent on the matrix size, while the algorithm \ref{alg:LDM_0T_linear} based on the banded structure has computation cost linearly dependent on the matrix size. As we can observed from Figure \ref{fig:LinearScaling} (a), the linear scaling algorithm can significantly reduce the computation time when $n$ is larger than certain moderate size. Similarly, we also conduct comparison between 
algorithm \ref{alg:LDM_FT} and algorithm \ref{alg:LDM_FT_Linear} for the finite temperature case, where matrices size has been chosen from $10^2$ to $2.5\times 10^3$. The slopes of the corresponding linear fitting curves also indicate that the proposed algorithm \ref{alg:LDM_FT_Linear} has linear scaling dependent on the problem size, while the original algorithm \ref{alg:LDM_FT} is cubically scaling to the problem size. Therefore, we can take the linear scaling advantage of the proposed algorithm \ref{alg:LDM_0T_linear} and algorithm \ref{alg:LDM_FT_Linear} based on the banded structure for problems with large size.

\begin{figure}[h]
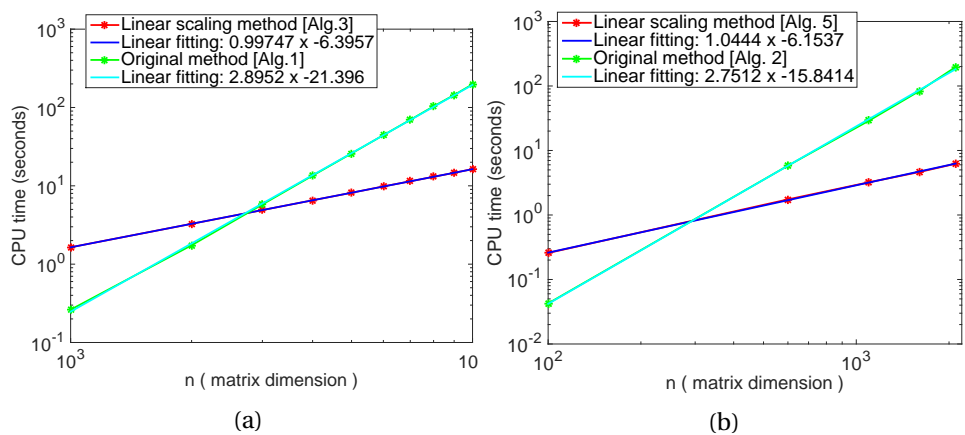

\begin{center}
\begin{minipage}{0.5\linewidth}
\includegraphics[width=1\linewidth]{TimeScaling_0T.eps}\\ 
\centering (a)
\end{minipage}\hfill
\begin{minipage}{0.5\linewidth}
\includegraphics[width=1\linewidth]{TimeScaling_FT.eps}\\ 
\centering (b)
\end{minipage}\hfill
\end{center}
\caption{(a): Zero temperature case: time consumption of each iteration used in Alg. \ref{alg:LDM_0T} v.s. Alg. \ref{alg:LDM_0T_linear}. (b): Finite temperature case: time consumption of each iteration used in Alg. \ref{alg:LDM_FT} v.s. Alg. \ref{alg:LDM_FT_Linear}}
\label{fig:LinearScaling}
\end{figure}


\section{Conclusion and future works}
\label{sec:conclusion}
This work extends our previous work \cite{LaiLuOsher:15} for constructing LDMs to the finite temperature case by adding an $\ell_1$ regularization to the energy of the original quantum system. As it has been shown that the density matrix decays exponential away from the diagonal for insulating system or system at finite temperature, the LDMs obtained by the proposed convex variational models provide good approximation to the original density matrices. Theoretically, we have conducted analysis to the approximation behavior.  In addition, we also design convergence guarantteed numerical algorithms to solve the proposed localized density matrix models based on Bregman iteration. More importantly, by observing that the $\ell_1$ regularized system can naturally create a localized density matrix with banded structure, we design approximate algorithms to find the localized density matrices. These numerical algorithms for banded matrices have computation complexity linear scaling to the matrix size $n$. 

This paper mainly focuses on proposing variational methods for LDMs,
theoretically analyzing its properties and designing numerical
algorithms. Thus, we only test some simple examples in the
experimental part to illustrate the proposed modes and algorithms. In
particular, we have taken a naive finite difference
discretization. All codes are implemented in \textsf{MATLAB} for test
purpose and are not optimized. In our future work, we will further
accelerate the computation from several aspects. First, we can
discretize the system using better localized basis to reduce the
problem size \cites{FHI-aims, LinLuYingE:DG, Siesta}. Second, we can
improve the implementation for better computation performance. In
addition, we will explore theoretical analysis for decay properties of
the LDM in future works.

\bibliographystyle{amsxport}
\bibliography{LocalizeMatrices}

\end{document}